\documentclass[12pt]{amsart}

\theoremstyle{plain}
\newtheorem{theorem}{Theorem}[section]

\newtheorem{corollary}[theorem]{Corollary}

\newtheorem{lemma}[theorem]{Lemma}

\newtheorem{proposition}[theorem]{Proposition}

\newtheorem{problem}[theorem]{Problem}

\theoremstyle{definition}
\newtheorem{example}[theorem]{Example}

\newtheorem{construction}[theorem]{Construction}

\numberwithin{equation}{section}

\usepackage{fullpage}
\usepackage{amssymb}
\usepackage{tikz}

\def\ldiv{\backslash}
\def\rdiv{/}

\def\A{\mathbf A}
\def\B{\mathbf B}
\def\Z{\mathbb Z}
\def\img{\mathrm{Img}}
\def\TC{\mathrm{TC}}
\def\W{\mathcal{W}}
\def\V{\boldsymbol{V}}

\def\mlt#1{\mathrm{Mlt}(#1)}
\def\inn#1{\mathrm{Inn}(#1)}
\def\aut#1{\mathrm{Aut}(#1)}
\def\totmlt#1{\mathrm{TMlt}(#1)}
\def\totinn#1{\mathrm{TInn}(#1)}
\def\Ng#1{\mathrm{Ng}(\;#1\;)}
\def\Cg#1{\mathrm{Cg}(\;#1\;)}
\def\sym#1{\mathrm{Sym}(#1)}

\def\treeL#1#2#3#4#5{
    \begin{tikzpicture}
    [scale = 0.8, operation/.style={circle,fill=gray!10,inner sep=0pt,minimum size=5mm}, element/.style={}]
        \node (1) at (0,0) [operation] {#4};
        \node (2L) at (-1,-1) [operation] {#2} edge (1);
        \node (2R) at (1,-1) [element] {#5} edge (1);
        \node (3L) at (-2,-2) [element] {#1} edge (2L);
        \node (3R) at (0,-2) [element] {#3} edge (2L);
    \end{tikzpicture}
}

\def\treeR#1#2#3#4#5{
    \begin{tikzpicture}
    [scale = 0.8, operation/.style={circle,fill=gray!10,inner sep=0pt,minimum size=5mm}, element/.style={}]
        \node (1) at (0,0) [operation] {#2};
        \node (2L) at (-1,-1) {#1} edge (1);
        \node (2R) at (1,-1) [operation] {#4} edge (1);
        \node (3L) at (0,-2) {#3} edge (2R);
        \node (3R) at (2,-2) {#5} edge (2R);
    \end{tikzpicture}
}

\begin{document}

\title{Commutator theory for loops}

\author{David Stanovsk\'y}

\author{Petr Vojt\v{e}chovsk\'y}

\address[Stanovsk\'y, Vojt\v{e}chovsk\'y]{Department of Mathematics, University of Denver, 2360 S Gaylord St, Denver, Colorado 80208, U.S.A.}
\address[Stanovsk\'y]{Department of Algebra, Faculty of Mathematics and Physics, Charles University, Praha 8, Sokolovsk\'a 83, 186 75, Czech Republic}
\email[Stanovsk\'y]{stanovsk@karlin.mff.cuni.cz}
\email[Vojt\v{e}chovsk\'y]{petr@math.du.edu}

\begin{abstract}
Using the Freese-McKenzie commutator theory for congruence modular varieties as the starting point, we develop commutator theory for the variety of loops. The fundamental theorem of congruence commutators for loops relates generators of the congruence commutator to generators of the total inner mapping group. We specialize the fundamental theorem into several varieties of loops, and also discuss the commutator of two normal subloops.

Consequently, we argue that some standard definitions of loop theory, such as elementwise commutators and associators, should be revised and linked more closely to inner mappings. Using the new definitions, we prove several natural properties of loops that could not be so elegantly stated with the standard definitions of loop theory. For instance, we show that the subloop generated by the new associators defined here is automatically normal.
We conclude with a preliminary discussion of abelianess and solvability in loops.
\end{abstract}



\thanks{Research partially supported by the Simons Foundation Collaboration Grant 210176 to Petr Vojt\v{e}chovsk\'y.}

\maketitle

\section{Introduction}

The two primary influences on modern loop theory come from group theory and universal algebra, a fact that is reflected already in the definition of a loop. Using the group-theoretical approach, a \emph{loop} is a nonempty set $Q$ with identity element $1$ and with binary operation $\cdot$ such that for every $a$, $b\in Q$ the equations $a\cdot x=b$, $y\cdot a=b$ have unique solutions $x$, $y\in Q$. The implied presence of divisions is made explicit in the equivalent universal algebraic definition due to Evans \cite{E}: a \emph{loop} is a universal algebra $(Q,1,\cdot,\ldiv,\rdiv)$ satisfying the identities
$$x\cdot 1 = x = 1\cdot x,\quad x\ldiv(x\cdot y)=y,\quad x\cdot(x\ldiv y)=y,\quad (y\cdot x)/x=y,\quad (y/x)\cdot x=y.$$
It is not difficult to see that associative loops are precisely groups, where we write $x^{-1}y$ and $xy^{-1}$ in place of $x\ldiv y$ and $x\rdiv y$, respectively.

The most influential text on loop theory in the English-speaking world is the book of Bruck \cite{B}. Although its title ``A survey of binary systems'' and its opening chapters are rather encompassing, it focuses on and culminates in the study of Moufang loops, a variety of loops with properties close to groups. It is therefore natural that Bruck's definitions are rooted mostly in group theory. For instance,
a subloop $N$ of a loop $Q$ is said to be \emph{normal} in $Q$ if
\begin{displaymath}
    xN = Nx,\quad x(yN) = (xy)N,\quad N(xy) = (Nx)y
\end{displaymath}
for every $x$, $y\in Q$,
the \emph{center} $Z(Q)$ of $Q$ is defined as
\begin{displaymath}
    Z(Q) = \{a\in Q;\;ax{=}xa,\,a(xy){=}(ax)y,\,x(ay){=}(xa)y,\,x(ya){=}(xy)a\text{ for every }x,y\in Q\},
\end{displaymath}
and the elementwise \emph{commutator} $[x,y]$ and \emph{associator} $[x,y,z]$ are defined as the unique solutions to the equations
\begin{displaymath}
    xy = (yx)[x,y],\quad (xy)z = (x(yz))[x,y,z],
\end{displaymath}
respectively. The \emph{associator subloop} $A(Q)$ of $Q$ is the smallest normal subloop of $Q$ such that $Q/A(Q)$ is a group, or, equivalently, the smallest normal subloop of $Q$ containing all associators $[x,y,z]$ of $Q$. The \emph{derived subloop} $Q'$ of $Q$ is the smallest normal subloop of $Q$ such that $Q/Q'$ is an abelian group, or, equivalently, the smallest normal subloop of $Q$ containing all commutators $[x,y]$ and all associators $[x,y,z]$ of $Q$. With $Q^{[0]}=Q$, $Q^{[i+1]}=(Q^{[i]})'$, a loop is called \emph{solvable} if $Q^{[n]}=1$ for some $n$.

It is easy to see that in groups the above concepts specialize to the usual group-theoretical notions (the associator and the associator subloop being void). Bruck's deep results \cite{B} showed that the group-theoretical definitions are sensible in Moufang loops, as did Glauberman's extension of the Feit-Thompson Odd Order Theorem to Moufang loops \cite{G2}.

As it turns out, the normality, the center and the derived notion of central nilpotency are the correct concepts for loops even from the universal algebraic point of view. (For normality this was known already to Bruck. For the center and central nilpotency this is probably folklore, but since we were not able to find a proof in the literature, we present it at the end of this paper for the convenience of the reader.) It is therefore not surprising that central nilpotency played a prominent role in the development of loop theory, as witnessed by the $42$ papers listed in MathSciNet under primary classification $20$N$05$ and with one of the words ``nilpotent'', ``nilpotency'' or ``nilpotence'' in the title. We mention \cite{C,CD,G2,GW,JKV,M,NV,NVo,N,NR,S2} as a representative sample.

But the commutators and associators did not fare as well, and neither did the concept of solvability. The inadequacies of the elementwise associators were first pointed out by Leong \cite{L}; see below for more details. There is no established notion of commutator of two normal subloops and, in contrast to nilpotency, there are only $9$ papers on MathSciNet under $20$N$05$ and with one of the words ``solvable'', ``solvability'', ``soluble'' or ``solubility'' in the title.

We maintain that this is not a coincidence, but rather a consequence of the fact that the elementwise associators, the commutator theory and solvability were not well conceived in loop theory. This is somewhat surprising, since loops are known to be congruence modular (they possess a Mal'tsev term), the general commutator theory for congruence modular varieties \cite{FM} has been developed more than $25$ years ago, and, furthermore, the original impetus for the commutator theory came from an important work of Smith \cite{S}, who set out to understand abelianess (or, centrality, in his terms) in quasigroups, a variety closely related to loops. For more historical details concerning commutator theory, see \cite{FM}.

The Freese-McKenzie commutator theory has proved useful in so many applications (see \cite{MS} for a survey) that we have little doubt it is the correct setting for loops, too. In this paper we derive the commutator theory for loops, with the congruence commutators at its core. The results are summarized in Section \ref{Sc:summary}.

Standard references to loop theory include \cite{Be,B,P}. See \cite{Berg,BS} for an introduction to universal algebra and \cite{MS} for an introduction to congruence commutators.

\section{Summary of results}\label{Sc:summary}

\subsection*{Inner mappings}

Let $Q$ be a loop with identity element $1$. For every $x\in Q$ let $L_x$, $R_x$, $M_x:Q\to Q$ be the bijections defined by
\begin{displaymath}
    L_x(y)= xy,\qquad R_x(y) = yx,\qquad M_x(y) = y\ldiv x.
\end{displaymath}
The mappings $L_x$, $R_x$ are traditionally called \emph{left} and \emph{right translations}. The mappings $y\mapsto x\ldiv y$ and $y\mapsto y/x$ are
the respective inverses of $L_x$ and $R_x$. The mapping $y\mapsto x\rdiv y$ is the inverse of $M_x$, because $z = y\ldiv x$ iff $yz = x$ iff $y = x\rdiv z$.

Following the conventional definitions of loop theory, the left and right translations generate the \emph{multiplication group} $\mlt Q$ of $Q$, i.e.,
\begin{displaymath}
    \mlt{Q}=\langle L_x,R_x;\;x\in Q\rangle.
\end{displaymath}
The \emph{inner mapping group} $\inn{Q}$ of $Q$ is the stabilizer of $1$ in $\mlt{Q}$.

To bring the mappings $M_x$ into play, we introduce the \emph{total multiplication group} $\totmlt Q$ of $Q$ as
\begin{displaymath}
    \totmlt Q=\langle L_x,R_x,M_x;\;x\in Q\rangle.
\end{displaymath}
The \emph{total inner mapping group} $\totinn{Q}$ of $Q$ is the stabilizer of $1$ in $\totmlt{Q}$. (Belousov \cite{Be2} was probably the first to ever consider total multiplication groups and total inner mapping groups.)

Although we will carefully distinguish between $\inn{Q}$ and $\totinn{Q}$, we will call elements of both $\inn{Q}$ and $\totinn{Q}$ \emph{inner mappings}. Note that, unlike in groups, $\inn{Q}$ is not necessarily a subgroup of the automorphism group $\aut{Q}$. (Loops where $\totinn Q\leq\aut Q$ were investigated in \cite{Be2}.)

In Section \ref{Ssc:TotMlt}, we calculate two small sets of generators for $\totinn{Q}$, namely
\begin{displaymath}
    \totinn{Q} = \langle L_{x,y},\, R_{x,y},\, M_{x,y},\, T_x,\, U_x;\;x,\,y\in Q\rangle
    = \langle A_{x,y}^\cdot,\,B_{x,y}^\cdot,\,A_{x,y}^\ldiv;\;x,\,y\in Q\rangle,
\end{displaymath}
where
\begin{displaymath}
    L_{x,y} = L_{xy}^{-1}L_xL_y,\quad
    R_{x,y} = R_{yx}^{-1}R_xR_y,\quad
    M_{x,y} = M_{y\ldiv x}^{-1}M_x M_y,\quad
    T_x =R_x^{-1}L_x,\quad
    U_x = R_x^{-1}M_x,
\end{displaymath}
and
\begin{displaymath}
	(z\cdot x)\circ y=A_{x,y}^\circ(z)\cdot(x\circ y),\qquad y\circ (z\cdot x)=B_{x,y}^\circ(z)\cdot(y\circ x).
\end{displaymath}
for $\circ\in\{\cdot,\ldiv,\rdiv\}$.
Each of these generating sets is an example of a \emph{set of words} that \emph{generates total inner mapping groups} in all loops, a concept that is formally defined in Section \ref{Ssc:words}. Informally, the above mappings, applied to an argument $z$, can be seen as loop terms in variables $x$, $y$, $z$ which yield a generating set of $\totinn{Q}$ for any loop $Q$ upon substituting all elements of $Q$ for $x$ and $y$.

\subsection*{The commutator}

Let $\A$ be a universal algebra. The congruences of $\A$ form a lattice with largest element $1_\A =\A\times \A$ and smallest element $0_\A=\{(a,a);\;a\in \A\}$.

Let $\alpha$, $\beta$, $\delta$ be congruences of $\A$. We say that \emph{$\alpha$ centralizes $\beta$ over $\delta$}, and write $C(\alpha,\beta;\delta)$, if for every $(n+1)$-ary term operation $t$, every pair $a\,\alpha\,b$ and every $u_1\,\beta\,v_1$, $\dots$, $u_n\,\beta\,v_n$ we have
\begin{displaymath}
    t(a,u_1,\dots,u_n)\,\delta\,t(a,v_1,\dots,v_n)\quad\text{implies}\quad t(b,u_1,\dots,u_n)\,\delta\,t(b,v_1,\dots,v_n).
\end{displaymath}
This implication is referred to as the \emph{term condition} for $t$, or $\TC(t,\alpha,\beta,\delta)$.\footnote{From now on, we will use the word ``term'' for both term operations and terms.}

The \emph{commutator} of $\alpha$, $\beta$, denoted by $[\alpha,\beta]$, is the smallest congruence $\delta$ such that $C(\alpha,\beta;\delta)$.
The Freese-McKenzie monograph \cite{FM} developed the theory and applications of this congruence operation in congruence modular varieties.

It is not easy to work with $C(\alpha,\beta;\delta)$ because the term condition must be tested for every term. It is therefore by no means straightforward to specialize the theory of \cite{FM} into a particular variety. The fundamental result of our paper is a description of the commutator $[\alpha,\beta]$ in loops, involving only a few special terms, namely the terms resulting from any set of words that generates total inner mapping groups. We will write Cg($X$) for the congruence generated by $X$, and we will denote by $\bar u$ the $n$-tuple $u_1,\dots,u_n$, where we intentionally omit the usual enclosing parentheses. We also write $\bar u\,\beta\,\bar v$ instead of $u_1\,\beta\,v_1$, $\dots$, $u_n\,\beta\,v_n$.

\begin{theorem}[Fundamental Theorem of Commutator Theory in Loops]\label{Th:CongrComInf}\label{Th:Fundamental}
Let $\V$ be a variety of loops and $\mathcal W$ a set of words that generates total inner mapping groups in $\V$. Then
\begin{displaymath}
    [\alpha,\beta] = \Cg{(W_{\bar u}(a), W_{\bar v}(a) );\;W\in\mathcal W,\,1\,\alpha\,a,\,\bar u\,\beta\,\bar v}
\end{displaymath}
for any congruences $\alpha$, $\beta$ of any $Q\in\V$.
\end{theorem}

A particular generating set for $[\alpha,\beta]$ is obtained anytime a suitable generating set $\mathcal W$ is given, for instance the above-mentioned set $\mathcal W=\{L_{x,y}$, $R_{x,y}$, $T_x$, $M_{x,y}$, $U_x\}$ for the variety $\V$ of all loops. See Example \ref{Ex:GeneratingSets} for other options.

Notice that the definition of the commutator is asymmetric, and so is the generating set from Theorem \ref{Th:Fundamental}. Nevertheless, $[\alpha,\beta]=[\beta,\alpha]$ for any congruences $\alpha,\beta$ in any algebra in a congruence modular variety \cite{FM}. This is an important property, although we do not need it in the present paper.

The proof of Theorem \ref{Th:CongrComInf} is presented in Section \ref{Sc:Proof}, and it is based on the words $A_{x,y}^\circ$, $B_{x,y}^\circ$. Note that the words $A_{x,y}^\circ$ resemble associators and the words $B_{x,y}^\circ$ look like a combination of commutators and associators, except that the element $z$ is absorbed into both $A_{x,y}^\circ(z)$ and  $B_{x,y}^\circ(z)$.

The machinery of Theorem \ref{Th:Fundamental} can be used to obtain numerous descriptions of the commutator $[\alpha,\beta]$. Theorem \ref{Th:FundamentalFinite} strengthens Theorem \ref{Th:CongrComInf} in loops satisfying a finiteness condition. More efficient generating sets $\mathcal W$ for Theorem \ref{Th:CongrComInf} are investigated in Section \ref{Ss:Pruning}, with the results summarized in Corollary \ref{Cr:CongrCom}. Generating sets in terms of elementwise associators and commutators, a traditional approach of group theory and loop theory, are given in Corollary \ref{Cr:CongrComAC}. The normal subloop corresponding to the congruence commutator is studied in Section~\ref{Sc:SubloopComm}.

Further simplifications are possible in specific classes of loops---throughout the paper we focus on inverse property loops, commutative loops and groups. In Section \ref{Sc:Examples} we illustrate how Theorem \ref{Th:CongrComInf} and its corollaries can be used to calculate the commutator in concrete loops. We also provide examples that witness that our results are optimal in certain ways.

\subsection*{Elementwise associators and commutators}

Leong \cite{L} noticed that, indeed, the associator $[x,y,z]$ corrects for the lack of associativity in the equation $(xy)z = x(yz)$, but so do many other associators, for instance the associator $a^L(x,y,z)$ defined by
\begin{displaymath}
    (xy)z = (a^L(x,y,z)x)(yz),
\end{displaymath}
or the associator $b^L(x,y,z)$ defined by
\begin{displaymath}
    x(yz) = ((b^L(x,y,z)x)y)z.
\end{displaymath}
The advantage of these new associators is that they relate to inner mappings, namely, $a^L(x,y,z) = R_{z,y}(x)/x$, and $b^L(x,y,z) = R_{z,y}^{-1}(x)/x$. Leong proved that in every loop $Q$ the subloop $\langle a^L(x,y,z),b^L(x,y,z);\;x,y,z\in Q\rangle$ is normal, and hence equal to $A(Q)$. He also showed that if $Q$ is a Moufang loop then $A(Q) = \langle [x,y,z];\;x,y,z\in Q\rangle$. Covalschi and Sandu \cite{CS} recently introduced similar associators that can be used to generate $Q'$.

The difficulty lies in deciding which associators should be used. Our approach is systematic and is based on the idea that elementwise associators and commutators should follow naturally from the commutator theory for congruences. Upon separating the roles of commutators and associators, we present a systematic definition of all possible associators and commutators in any loop, the result being summarized in Table \ref{Tb:AC}. Importantly, all these associators and commutators are associated with certain inner mappings (they evaluate to $1$ when $x=1$ is substituted), and thus can be used to obtain a generating set of the congruence commutator; see Corollary \ref{Cr:CongrComAC}.

In Section \ref{Sc:ASDS}, we make a case for our commutators and associators. First, we show that $Q'$ is the subloop generated by a choice of associators and commutators whenever the corresponding inner mappings generate $\inn Q$; see Theorem \ref{Th:Q'}. Then, imitating the proof of Leong, we show that certain associators generate $A(Q)$; see Theorem \ref{Th:A(Q)}.

\subsection*{The commutator of normal subloops}

It is well known that a subloop of a loop $Q$ is normal iff it is a kernel of some homomorphism from $Q$ to another loop. Equivalently, normal subloops are precisely the blocks of congruences on $Q$ containing the identity element $1$, or subloops closed under all inner mappings from $\inn Q$. In Proposition \ref{Pr:TotInn}, we show that normal subloops are closed under all inner mappings from $\totinn Q$, too.

There exists an order-preserving correspondence between the lattice of normal subloops of $Q$ and the lattice of congruences of $Q$. If $N$ is a normal subloop of $Q$, let $\gamma_N$ be the congruence on $Q$ defined by
\begin{displaymath}
    a\,\gamma_N\,b\text{ iff }a/b\in N,
\end{displaymath}
or, equivalently, iff $b\ldiv a\in N$, $b/a\in N$, or $a\ldiv b\in N$. If $\alpha$ is a congruence of $Q$, let $N_\alpha$ be the normal subloop of $Q$ defined by
\begin{displaymath}
    N_\alpha = \{a\in Q;\;a\,\alpha\,1\}.
\end{displaymath}
For two normal subloops $A$, $B$ of $Q$, define the \emph{commutator} of $A$ and $B$ in $Q$ by
\begin{displaymath}
    [A,B]_Q=N_{[\gamma_A,\gamma_B]}.
\end{displaymath}
The above correspondence allows us to immediately translate Theorem \ref{Th:CongrComInf} from the language of congruences to the language of normal subloops. We will write Ng($X$) for the smallest normal subloop containing the set $X$, and $\bar u/\bar v\in B$ as a shorthand for $u_1/v_1\in B$, $\dots$, $u_n/v_n\in B$.

\begin{theorem}\label{Th:NormComInf}\label{Th:FundamentalS}
Let $\V$ be a variety of loops and $\mathcal W$ a set of words that generates total inner mapping groups in $\V$. Then
\begin{displaymath}
    [A,B]_Q = \Ng{W_{\bar u}(a)/W_{\bar v}(a);\;W\in\mathcal W,\,a\in A,\,\bar u/\bar v\in B }
\end{displaymath}
for any normal subloops $A$, $B$ of any $Q\in\V$.
\end{theorem}

Using the ideas of Section \ref{Sc:AC}, we can choose $\W$ so that we can interpret the generating set of $[A,B]_Q$ as quotients of associators and commutators.

In Section \ref{Sc:SubloopComm}, we explore simplifications of the generating set. It so happens that in groups the normal closure is not needed and the quotients can be reduced, i.e.,
\begin{displaymath}
    [A,B]_Q = \langle [a,u]/[a,v];\;a\in A,\,u/v\in B\rangle = \langle [a,b];\;a\in A,\,b\in B\rangle.
\end{displaymath}
Neither of these properties holds in general loops, as illustrated by examples in Section \ref{Sc:Examples}. The normal closure can be avoided in all \emph{automorphic loops}, that is, loops $Q$ with $\inn{Q}\le\aut{Q}$; see Proposition \ref{Pr:automorphic}. Some types of quotients can be reduced in all loops; see Corollary \ref{Cr:FundamentalS}.

\subsection*{Center and nilpotency, abelianess and solvability}

The general commutator theory for universal algebras offers more than just the commutator of two congruences. An algebra $\A$ is called \emph{nilpotent}, if $\gamma_{(n)}=0_\A$ for some $n$, where
\begin{displaymath}
    \gamma_{(0)}=1_\A,\qquad \gamma_{(i+1)}=[\gamma_{(i)},1_\A].
\end{displaymath}
An algebra $\A$ is called \emph{solvable}, if $\gamma^{(n)}=0_\A$ for some $n$, where
\begin{displaymath}
    \gamma^{(0)}=1_\A,\qquad \gamma^{(i+1)}=[\gamma^{(i)},\gamma^{(i)}].
\end{displaymath}
Notice that both definitions use a special type of commutators: nilpotency requires only commutators $[\alpha,1_\A]$, while solvability requires only commutators $[\alpha,\alpha]$. Both of these types of commutators can be defined using specialized concepts: center and abelianess.

Let $\A$ be an algebra. The \emph{center} of $\A$, denoted by $\zeta(\A)$, is the largest congruence of $\A$ such that $C(\zeta(\A),1_\A;0_\A)$. It is easy to show that $[\alpha,1_\A]$ is the smallest congruence $\delta$ such that $\alpha/\delta\leq\zeta(\A/\delta)$.

A congruence $\alpha$ of an algebra $\A$ is called \emph{abelian} if $C(\alpha,\alpha;0_\A)$. It is easy to show that $[\alpha,\alpha]$ is the smallest congruence $\delta$ such that $\alpha/\delta$ is an abelian congruence of $\A/\delta$. An algebra $\A$ is called \emph{abelian} if $\zeta(\A)=1_\A$, or, equivalently, if the congruence $1_\A$ is abelian.

An argument similar to the one in group theory shows that $\A$ is nilpotent (resp. solvable) if and only if there is a chain of congruences \begin{displaymath}
    1_\A=\alpha_0\geq\alpha_1\geq\ldots\geq\alpha_n=0_\A
\end{displaymath}
such that $\alpha_{i}/\alpha_{i+1}\leq\zeta(\A/\alpha_{i+1})$ (resp. such that $\alpha_{i}/\alpha_{i+1}$ is an abelian congruence of $\A/\alpha_{i+1}$) for all $i\in\{0,1,\dots,n-1\}$.

Now, let $Q$ be a loop. One can quickly show (and it follows from Theorem \ref{Th:Center}) that a loop is abelian if and only if it is a commutative group. With respect to nilpotency and solvability, there are good news and bad news.

Fortunately, the center $\zeta(Q)$ as defined in universal algebra, and the center $Z(Q)$ as defined in loop theory agree, i.e., $N_{\zeta(Q)}=Z(Q)$; see Theorem \ref{Th:Center}. Consequently, nilpotency based on the commutator theory is the same concept as central nilpotency traditionally used in loop theory. In our opinion, this explains why central nilpotency has been playing a prominent role in loop theory.

Unfortunately, Bruck's concept of solvability derived from group theory does not agree with the universal algebraic solvability. The commutator theory suggests there is a difference between \emph{abelianess} of an algebra and \emph{abelianess in an algebra}. This is inherent in the congruence approach, since congruences carry over the universe of the original algebra, so the congruence commutator $[\alpha,\alpha]$ automatically takes place in the underlying algebra $\A$. In loops, we have to be careful. Upon translating the concept of abelianess from congruences to normal subloops, we note that a normal subloop $N$ of $Q$ is \emph{abelian} if $[1_N,1_N]=0_N$. However, $N$ is \emph{abelian in} $Q$ if $[\gamma_N,\gamma_N]=0_Q$, a stronger property in general. Examples of an abelian loop $N\unlhd Q$ that is not abelian in $Q$ were known already to Freese and McKenzie; see \cite[Chapter 5, Exercise 10]{FM} or our Example \ref{Ex:Z4}.

In our opinion, this explains why there are relatively few results on solvable loops, and why most existing results deal with varieties of loops that are close to groups. For instance, the Feit-Thompson Odd Order Theorem \cite{FT} has been extended from groups to Moufang loops in \cite{G2}, and to automorphic loops in \cite{KKPV}, Hall's theorem for Moufang loops can be found in \cite{G2} and in \cite{Ga}. A notable exception is the general result of Vesanen \cite{V}: if the group $\mlt{Q}$ is solvable then $Q$ itself is solvable in the group-theoretical sense. But hardly anything is known in the other direction, starting with the assumption that $Q$ is solvable. Could this be so because the traditional definition of solvability in loops is too weak?

We propose to call a loop $Q$ \emph{congruence solvable} if there is a chain $1=Q_0 \le Q_1 \le \cdots \le Q_n = Q$ of normal subloops $Q_i$ of $Q$ such that every factor $Q_{i+1}/Q_i$ is abelian in $Q/Q_i$. Does congruence solvability of $Q$ relate to the structure of the total multiplication group $\totmlt{Q}$? Is group-theoretical solvability equivalent to congruence solvability in classes of loops close to groups? These questions and related problems are subject of an ongoing investigation of the authors. Here we at least show that while every congruence solvable loop is indeed solvable, the converse is not true; see Example \ref{Ex:Z4}.

Section \ref{Sc:univalg} explains in more detail how the center, central nilpotency, abelianess and solvability specialize from universal algebras to loops, and from loops to groups. We also provide references to other alternative approaches to nilpotency and solvability in loops and other classes of algebras.

\subsection*{Auxiliary definitions}

Let $Q$ be a loop with \emph{two-sided inverses}, that is, a loop in which for every $x\in Q$ there is $x^{-1}\in Q$ such that $xx^{-1}= x^{-1}x=1$. We then define the inversion mapping
\begin{displaymath}
    J:Q\to Q, \qquad J(x)=x^{-1}.
\end{displaymath}
Note that $(x^{-1})^{-1}=x$, and hence $J$ is involutory.

We say that $Q$ has the \emph{anti-automorphic inverse property} (AAIP) if $(xy)^{-1}=y^{-1}x^{-1}$ for every $x,y\in Q$, or, equivalently, if the mapping $J$ is an anti-automorphism.
We say that $Q$ has the \emph{inverse property} if $x^{-1}(xy)=y = (yx)x^{-1}$ holds for every $x$, $y\in Q$. Then $x\ldiv y$ can be replaced by $x^{-1}y$ and $x/y$ by $xy^{-1}$, as in groups. Note that inverse property loops have the AAIP: $(xy)\cdot (xy)^{-1}x = x = (xy)y^{-1}$, so $(xy)^{-1}x = y^{-1}$, and using the inverse property again, $(xy)^{-1} = y^{-1}x^{-1}$. Many highly structured varieties of loops have the inverse property, most notably groups and Moufang loops.

\section{Inner mappings}\label{Sc:TotMlt}

\subsection{Inner mapping groups}

Let $Q$ be a loop. Recall that $\mlt Q=\langle L_x,R_x;\;x\in Q\rangle$ and $\inn Q=\mlt Q_1 = \{f\in\mlt Q;\; f(1)=1\}$. One possible generating set of $\inn{Q}$ is described in the well known Proposition \ref{Pr:InnGens}, which is in turn based on a variation of a result of O.~Schreier about generators of stabilizers.

\begin{lemma}[O.~Schreier]\label{Lm:StabilizerGens}
Let $G$ be a transitive permutation group on a set $X$ and let $c\in X$. For $y\in X$ let $g_y\in G$ be such that $g_y(c)=y$, where we choose $g_c=1$. If $G=\langle H\rangle$ then $G_c = \langle g_{h(y)}^{-1}hg_y;\;h\in H,\,y\in X\rangle.$
\end{lemma}
\begin{proof}
Let $g\in G_c$. Since $G=\langle H\rangle$, there are $h_1$, $\dots$, $h_n\in H\cup H^{-1}$ such that $g = h_1\cdots h_n$. We thus have
\begin{displaymath}
    g = g_{h_1\cdots h_n(c)}(g_{h_1\cdots h_n(c)}^{-1}h_1g_{h_2\cdots h_n(c)})\cdots
    (g_{h_{n-1}h_n(c)}^{-1}h_{n-1}g_{h_n(c)})(g_{h_n(c)}^{-1}h_ng_c)g_c^{-1}.
\end{displaymath}
Note that $g_{h_1\cdots h_n(c)} = g_{g(c)} = g_c = 1$, so $g$ is a product of elements of the form $g_{h(y)}^{-1}hg_y$, where $h\in H\cup H^{-1}$. The identity
\begin{displaymath}
    (g_{h(y)}^{-1}hg_y)^{-1} = g_y^{-1}h^{-1}g_{h(y)} = g^{-1}_{h^{-1}h(y)}h^{-1}g_{h(y)}
\end{displaymath}
shows that generators of the form $g^{-1}_{h(y)}hg_y$ with $h\in H$ suffice.
\end{proof}

Recall the mappings $A_{x,y}^\cdot=R_{xy}^{-1}R_yR_x$ and $B_{x,y}^\cdot=R_{yx}^{-1}L_yR_x$ introduced in Section \ref{Sc:summary}.

\begin{proposition}\label{Pr:InnGens}
Let $Q$ be a loop. Then
\begin{displaymath}
    \inn{Q} = \langle L_{x,y},\,R_{x,y},\,T_x;\;x,y\in Q\rangle = \langle A_{x,y}^\cdot,B_{x,y}^\cdot;\;x,y\in Q\rangle.
\end{displaymath}
\end{proposition}
\begin{proof}
The multiplication group $G=\mlt{Q}$ acts transitively on $X=Q$. Upon applying Lemma \ref{Lm:StabilizerGens} with $c=1$, $g_y = R_y$ and $H = \{L_x,\,R_x;\;x\in Q\}$, we conclude that $G_1 = \inn{Q}$ is generated by $\{R_{L_x(y)}^{-1}L_xR_y,\,R_{R_x(y)}^{-1}R_xR_y;\;x,\,y\in Q\} = \{B_{y,x}^\cdot,\,A_{y,x}^\cdot;\;x,y\in Q\}$. Now note that $A_{y,x}^\cdot=R_{x,y}$ and $B_{y,x}^\cdot = R_{L_x(y)}^{-1}L_xR_y  = (R_{xy}^{-1}L_{xy})(L_{xy}^{-1}L_xL_y)(L_y^{-1}R_y) = T_{xy}L_{x,y}T_y^{-1}$.
\end{proof}



An immediate corollary of Proposition \ref{Pr:InnGens} is the observation that $\inn{Q}=1$ if and only if $Q$ is an abelian group. We will need this fact in Section \ref{Sc:ASDS}.

The significance of $\inn{Q}$ is that it can be used to characterize normal subloops, just as in the case of groups.

\begin{proposition}\label{Pr:Inn}
Let $N$ be a subloop of $Q$. Then the following conditions are equivalent:
\begin{enumerate}
\item[(i)] $N$ is normal in $Q$, that is, $xN=Nx$, $x(yN) = (xy)N$, $N(xy) = (Nx)y$ holds for every $x$, $y\in Q$.
\item[(ii)] $f(N)=N$ for every $f\in\inn{Q}$.
\item[(iii)] $N$ is the kernel of some loop homomorphism.
\item[(iv)] $N$ is the block containing $1$ of some congruence on $Q$.
\end{enumerate}
\end{proposition}
\begin{proof}
This is folklore. See \cite[Section I.7]{P} for the equivalence of (i)--(iii), or \cite{DKV}.
\end{proof}

\subsection{Total inner mapping groups}\label{Ssc:TotMlt}

In this subsection, we partly follow Belousov and Shcherbakov \cite{Be2,Sh}.
Recall that $\totmlt Q=\langle L_x,R_x,M_x;\;x\in Q\rangle$, where $M_x(y)=y\ldiv x$, and $\totinn Q=\totmlt Q_1$.
Also recall the mappings
\begin{displaymath}
    M_{x,y} = M_{y\ldiv x}^{-1}M_x M_y\quad\text{and}\quad U_x = R_x^{-1}M_x
\end{displaymath}
and note that $M_{x,y}$, $U_x\in\totinn{Q}$ thanks to $M_{x,y}(1) = (y\ldiv x)\rdiv ((1\ldiv y)\ldiv x) = 1$ and $U_x(1) = (1\ldiv x)\rdiv x = 1$. Finally, the mappings $A_{x,y}^\circ$, $B_{x,y}^\circ$ can be written as
\begin{align*}
    A_{x,y}^\cdot&=R_{xy}^{-1}R_yR_x,&  \qquad B_{x,y}^\cdot &=R_{yx}^{-1}L_yR_x, \\
    A_{x,y}^\ldiv&=R_{x\ldiv y}^{-1}M_yR_x,&  \qquad B_{x,y}^\ldiv &=R_{y\ldiv x}^{-1}L_y^{-1}R_x, \\
    A_{x,y}^\rdiv&=R_{x\rdiv y}^{-1}R_y^{-1}R_x,&  \qquad B_{x,y}^\rdiv &=R_{y\rdiv x}^{-1}M_y^{-1}R_x,
\end{align*}
and we note that $A_{x,y}^\rdiv = (A_{x\rdiv y,y}^\cdot)^{-1}$, $B_{x,y}^\ldiv = (B_{y\ldiv x,y}^\cdot)^{-1}$ and $B_{x,y}^\rdiv = (A_{y\rdiv x,y}^\ldiv)^{-1}$. All these mappings fix the identity element $1$ and hence are inner mappings.

\begin{proposition}\label{Pr:TotInnGens}
Let $Q$ be a loop. Then
\begin{displaymath}
    \totinn{Q} = \langle L_{x,y},\,R_{x,y},\,M_{x,y},\,T_x,\,U_x;\;x,y\in Q\rangle
    = \langle A_{x,y}^\cdot,\,B_{x,y}^\cdot,\,A_{x,y}^\ldiv;\;x,y\in Q\rangle.
\end{displaymath}
\end{proposition}
\begin{proof}
Let us apply Lemma \ref{Lm:StabilizerGens} to the transitive group $G = \totmlt{Q}$ with $X=Q$, $c=1$, $g_y=R_y$ and $H = \{L_x,\,R_x,\,M_x;\;x\in Q\}$. We conclude that $G_1 = \totinn{Q} = \langle R_{L_x(y)}^{-1}L_xR_y,\,R_{R_x(y)}^{-1}R_xR_y,\,R_{M_x(y)}^{-1}M_xR_y;\;x,y\in Q\rangle = \langle B_{y,x}^\cdot,\,A_{y,x}^\cdot,\,A_{y,x}^\ldiv;\;x,y\in Q\rangle$. By Proposition \ref{Pr:InnGens}, $\inn{Q} = \langle A_{x,y}^\cdot,\,B_{x,y}^\cdot;\;x,y\in Q\rangle = \langle L_{x,y},\,R_{x,y},\,T_x;\;x,y\in Q\rangle$. We finish with $A_{y,x}^\ldiv = R_{y\ldiv x}^{-1}M_xR_y = (R_{y\ldiv x}^{-1}M_{y\ldiv x})(M_{y\ldiv x}^{-1}M_xM_y)(M_y^{-1}R_y) = U_{y\ldiv x}M_{x,y}U_y^{-1}$.
\end{proof}

\begin{problem}
Are the generating sets from Proposition \ref{Pr:TotInnGens} minimal, in the sense that none of the five (respectively three) types of mappings can be removed? Is there a generating set for $\totmlt Q$ with only two types of inner mappings?
\end{problem}

\begin{example}\label{Ex:TotInnGens}
Consider the loops $Q_1$, $Q_2$, $Q_3$, $Q_4$ with multiplication tables
\begin{displaymath}
    \begin{array}{c|cccccc}
        Q_1&1&2&3&4&5&6\\
        \hline
 1&1& 2& 3& 4& 5& 6 \\ 2&2& 1& 4& 3& 6& 5 \\ 3&3& 4& 5& 6& 2& 1 \\ 4&4& 6& 2& 5& 1& 3 \\ 5&5& 3& 6& 1& 4& 2 \\ 6&6& 5& 1& 2& 3& 4
    \end{array}
    \qquad\qquad
    \begin{array}{c|cccccc}
        Q_3&1&2&3&4&5&6\\
        \hline
 1&1& 2& 3& 4& 5& 6 \\ 2&2& 1& 4& 5& 6& 3 \\ 3&3& 4& 5& 6& 1& 2 \\ 4&4& 5& 6& 3& 2& 1 \\ 5&5& 6& 1& 2& 3& 4 \\ 6&6& 3& 2& 1& 4& 5
    \end{array}
\end{displaymath}
\begin{displaymath}
    \begin{array}{c|cccccccc}
        Q_2&1&2&3&4&5&6&7&8\\
        \hline
 1&1& 2& 3& 4& 5& 6& 7& 8 \\ 2&2& 1& 4& 3& 7& 8& 5& 6 \\ 3&3& 4& 1& 2& 6& 5& 8& 7 \\ 4&4& 3& 2& 1& 8& 7& 6& 5 \\
   5&5& 6& 7& 8& 1& 2& 3& 4 \\ 6&6& 8& 5& 7& 3& 1& 4& 2 \\ 7&7& 5& 8& 6& 2& 4& 1& 3 \\ 8&8& 7& 6& 5& 4& 3& 2& 1
    \end{array}
    \qquad\qquad
    \begin{array}{c|cccccccc}
        Q_4&1&2&3&4&5&6&7&8\\
        \hline
 1&1& 2& 3& 4& 5& 6& 7& 8 \\ 2&2& 1& 4& 3& 6& 5& 8& 7 \\ 3&3& 4& 1& 2& 7& 8& 5& 6 \\ 4&4& 3& 2& 1& 8& 7& 6& 5 \\
   5&5& 6& 7& 8& 1& 2& 4& 3 \\ 6&6& 5& 8& 7& 2& 1& 3& 4 \\ 7&7& 8& 5& 6& 4& 3& 1& 2 \\ 8&8& 7& 6& 5& 3& 4& 2& 1
    \end{array}
\end{displaymath}
In the \texttt{GAP} package \texttt{LOOPS}, these are the loops with catalog numbers $Q_1 = \texttt{SmallLoop(6,8)}$, $Q_2 = \texttt{LeftBolLoop(8,1)}$, $Q_3 = \texttt{SmallLoop(6,47)}$ and $Q_4 = \texttt{AutomorphicLoop(8,1)}$. Then it can be verified in \texttt{GAP} that
\begin{displaymath}
    \langle L_{x,y},\,R_{x,y},\,M_{x,y},\,U_x;\;x,y\in Q\rangle\neq\totmlt Q
\end{displaymath}
for $Q=Q_1,Q_2$, and
\begin{displaymath}
    \langle L_{x,y},\,R_{x,y},\,M_{x,y},\,T_x;\;x,y\in Q\rangle\neq\totmlt Q\neq\langle  A_{x,y}^\cdot,\,B_{x,y}^\cdot;\;x,\,y\in Q\rangle
\end{displaymath}
for $Q=Q_3,Q_4$. Hence neither of the mappings $T_x$, $U_x$, $A_{x,y}^\ldiv$ can be removed in general from the generating sets of $\totinn{Q}$ (cf. Proposition \ref{Pr:TotInnGens}), even in some highly structured varieties of loops.
\end{example}

We now observe that $\totinn Q$ can also be used to characterize normal subloops, hence adding another equivalent condition to Proposition \ref{Pr:Inn}.

\begin{proposition}[\cite{Be2}]\label{Pr:TotInn}
Let $N$ be a subloop of $Q$. Then $N$ is normal in $Q$ if and only if $f(N)=N$ for every $f\in\totinn{Q}$.
\end{proposition}

\begin{proof}
In view of Propositions \ref{Pr:Inn} and \ref{Pr:TotInnGens}, it only remains to show that if $N\unlhd Q$, then $U_x(a) = (a\ldiv x)\rdiv x\in N$ and $M_{x,y}(a) = (y\ldiv x)\rdiv ((a\ldiv y)\ldiv x) \in N$ for every $x$, $y\in Q$ and $a\in N$. Let $\varphi$ be a homomorphism from $Q$ to another loop such that $N=\ker(\varphi)$. Then $\varphi(U_x(a)) = (\varphi(a)\ldiv \varphi(x))\rdiv \varphi(x) = U_{\varphi(x)}(\varphi(a)) = U_{\varphi(x)}(1)=1$, and, similarly, $\varphi(M_{x,y}(a)) = M_{\varphi(x),\varphi(y)}(\varphi(a)) = M_{\varphi(x),\varphi(y)}(1) = 1$.
\end{proof}

We now focus on an important special case, the class of inverse property loops (see also \cite{Sh}).

\begin{proposition}\label{Pr:TotIP}
Let $Q$ be an inverse property loop. Then
\begin{enumerate}
	\item[(i)] $\totmlt Q=\langle L_x,\,J;\;x\in Q\rangle=\langle M_x;\;x\in Q\rangle$.
	\item[(ii)] $\totinn Q=\langle  L_{x,y},\,T_x,\,J;\;x,y\in Q\rangle=\langle  M_{x,y};\;x,y\in Q\rangle$.
\end{enumerate}
\end{proposition}
\begin{proof}
(i) Clearly, $J\in\totmlt Q$, since $J=M_1$. To show that $L_x,J$ generate $\totmlt Q$, observe that $R_x=JL_x^{-1}J$ and $M_x=JL_x^{-1}$. To show that $M_x$ generate $\totmlt Q$, observe again that $J=M_1$ and $L_x=M_x^{-1}J$.

(ii) For the first assertion, apply Lemma \ref{Lm:StabilizerGens} to the transitive group $G = \totmlt{Q}$ with $X=Q$, $c=1$, $g_y=L_y$ and $H = \{L_x,\,J;\;x\in Q\}$. We conclude that $G_1 = \totinn{Q} = \langle L_yJL_y,\,L_{xy}^{-1}L_xL_y;\;x,y\in Q\rangle = \langle JT_y,\,L_{x,y};\;x,y\in Q\rangle=\langle L_{x,y},\,T_x,\,J;\;x,y\in Q\rangle$, because $J=JT_1$.

For the second assertion, apply Lemma \ref{Lm:StabilizerGens} to the transitive group $G = \totmlt{Q}$ with $X=Q$, $c=1$, $g_y=R_y$ and $H = \{M_x,;\;x\in Q\}$. We conclude that $G_1 = \totinn{Q} = \langle R_{y\ldiv x}^{-1}M_xR_y;\;x,y\in Q\rangle = \langle JM_{x,y}J;\;x,y\in Q\rangle=\langle M_{x,y};\;x,y\in Q\rangle$, because $J=M_{1,1}$ and $JM_{1,1}J=J^3=J$.
\end{proof}

We finish this subsection with a side remark. It is well known that the multiplication group $\mlt{G}$ of a group $G$ is isomorphic to $(G\times G)/\{(a,a);\;a\in Z(G)\}$. Here is an analogous description of $\totmlt{G}$:

\begin{proposition}
Let $G$ be a group.
\begin{enumerate}
\item[(i)] If $G$ is not an elementary abelian $2$-group then $\totmlt{G}$ is isomorphic to
\begin{displaymath}
    ((G\times G)\rtimes \Z_2)/\{(a,a,0);\;a\in Z(G)\},
\end{displaymath}
where $\Z_2$ acts on $G\times G$ by transposing the two coordinates in the direct product.
\item[(ii)] If $G$ is an elementary abelian $2$-group then $\totmlt{G} = \mlt{G}$ is isomorphic to $G$.
\end{enumerate}
\end{proposition}
\begin{proof}
Note that $\totmlt{G} = \langle L_x,R_x,J;\;x\in G\rangle$. If $G$ is an elementary abelian $2$-group then $J$ is the identity mapping, $L_x=R_x$ for every $x\in G$, and hence $\totmlt{G}=\{L_x;\;x\in G\}$ is isomorphic to $G$ according to Cayley's left regular representation. For the rest of the proof assume that $G$ is not an elementary abelian $2$-group.

With the action from the statement of the proposition, the multiplication in $(G\times G)\rtimes \Z_2$ is given by
\begin{displaymath}
    (a,b,u)(c,d,v) = \left\{\begin{array}{ll}
        (ac,bd,v),&\text{ if $u=0$},\\
        (ad,bc,1+v),&\text{ if $u=1$}.
    \end{array}\right.
\end{displaymath}
Define $\varphi:(G\times G)\rtimes\Z_2\to \totmlt{G}$ by $\varphi(a,b,u)=L_aR_b^{-1}J^u$. To check that $\varphi$ is a homomorphism, take $\varphi(a,b,u)\varphi(c,d,v)(x)=L_aR_b^{-1}J^uL_cR_d^{-1}J^v(x)$ and consider two cases. If $u=0$, the above element is equal to $acJ^v(x)d^{-1}b^{-1} = L_{ac}R_{bd}^{-1}J^{v}(x) = \varphi(ac,bd,v)(x) = \varphi((a,b,0)(c,d,v))(x)$. If $u=1$, we calculate $a(cJ^v(x)d^{-1})^{-1}b^{-1} = adJ^{1+v}(x)c^{-1}b^{-1} = L_{ad}R_{bc}^{-1}J^{1+v}(x) = \varphi(ad,bc,1+v)(x)=\varphi((a,b,1)(c,d,v))(x)$.

Since the image of $\varphi$ contains all generators of $\totmlt{G}$, we see that $\varphi$ is onto $\totmlt{G}$. The kernel of $\varphi$ consists of all $(a,b,u)$ such that $L_aR_b^{-1}J^u$ is the identity mapping, which means $aJ^u(x)b^{-1}=x$ for every $x\in G$. If $x=1$, we obtain $a=b$, hence the kernel only contains triples $(a,a,u)$ such that $aJ^u(x)a^{-1}=x$ for every $x\in G$, or, equivalently, $a^{-1}xa=J^u(x)$ for every $x\in G$. If $u=0$, this is equivalent to $a\in Z(G)$. If $u=1$, the left hand side defines an automorphism of $G$, but $J$ is an automorphism only if $G$ is abelian. When $G$ is abelian, the condition $x=a^{-1}xa=J(x)=x^{-1}$ says that $G$ is an elementary abelian $2$-group, a contradiction. Hence $\ker(\varphi) = \{(a,a,0);\;a\in Z(G)\}$.
\end{proof}

\subsection{Generating inner mappings uniformly}\label{Ssc:words}

Propositions \ref{Pr:InnGens}, \ref{Pr:TotInnGens} and \ref{Pr:TotIP} establish small sets of inner mappings generating $\inn Q$ and $\totinn Q$ for a loop $Q$ in certain varieties (of all loops, and all IP loops). Interestingly, these sets generate the respective groups \emph{uniformly}, independently of $Q$ within a given variety. In the rest of the section, we will formalize this idea, to be used in the proof of Theorem \ref{Th:Fundamental}.

Formally, a \emph{word} $W=W_{\bar x}$ is an element of the free group generated by letters $K_t$, where $K\in\{L,R,M\}$ and $t$ is a term over $\bar x$. Equivalently, it is a formal expression of the form $K^1_{t_1(\bar x)}\dots K^k_{t_k(\bar x)}$ where $K^1,\dots,K^k$ are letters from $\{L,R,M,L^{-1},R^{-1},M^{-1}\}$ and $t_1,\dots,t_k$ are arbitrary loop terms.
Every word induces a mapping
\begin{displaymath}
    W_{\bar x}:Q^n\to\totmlt Q, \qquad \bar a\mapsto W_{\bar a},
\end{displaymath}
where $W_{\bar a}$ is the mapping obtained by replacing every $K^i_{t_i(\bar a)}$ with the actual translation on $Q$ by $t_i(\bar a)$. We can thus write $W_{\bar a}(b)\in Q$ for the result of the mapping $W_{\bar a}$ on $b\in Q$.
Every word $W$ also induces a term: given another variable $y$, we can consider $W_{\bar x}(y)$ as a term, resulting by evaluating the mapping $W_{\bar x}$ in the free loop of terms over $x_1,\dots,x_n,y$. The following examples should make this clear.

\begin{example}
The expression $L_{x,y}$ defined by $L_{xy}^{-1}L_xL_y$ is a word. Then $L_{x,y}(z)$ denotes the term $(xy)\ldiv(x(yz))$, and for every loop $Q$ and every $a,b\in Q$, $L_{a,b}$ denotes the inner mapping $L_{ab}^{-1}L_aL_b$. Note that $J$ is also a word, with no parameters, defined by $J=M_1$, where 1 is a term with no parameters.
\end{example}

Let $\V$ be a variety of loops. We say that a word $W$ is \emph{inner} for $\V$, if $W_{\bar a}\in\totinn Q$ for every $Q\in\V$ and every $a_1,\dots,a_n\in Q$. Note that this is equivalent to saying that $W_{\bar x}(1)=1$ is a valid identity in the variety $\V$. Again, let us clarify the idea with an example.

\begin{example}
The word $L_{x,y}$ is inner for the variety of all loops. The word $J$ is inner for the variety of all inverse property loops. The word $W=L_xL_x$ is inner for the variety of all loops satisfying the identity $x^2=1$, since $L_aL_a(1)=a^2=1$, but it is not inner for the variety of all loops.
\end{example}

Let $\V$ be a variety of loops, and let $\W$ be a set of inner words for $\V$. We say that $\W$ \emph{generates inner mapping groups} in $\V$ if for every $Q\in\V$ we have
\begin{displaymath}
    \inn Q=\langle W_{\bar a};\;W\in\W,\,a_1,\dots,a_n\in Q\rangle,
\end{displaymath}
and it \emph{generates total inner mapping groups} in $\V$ if for every $Q\in\V$ we have
\begin{displaymath}
    \totinn Q=\langle W_{\bar a};\;W\in\W,\,a_1,\dots,a_n\in Q\rangle.
\end{displaymath}

\begin{example}\label{Ex:GeneratingSets}
Using Propositions \ref{Pr:InnGens}, \ref{Pr:TotInnGens}, \ref{Pr:TotIP}, and some trivial observations, we have:
\begin{enumerate}
\item[$\bullet$]
    Let $\V$ be a variety of loops. Then $\{L_{x,y},\,R_{x,y},\,T_x\}$ generates inner mapping groups in $\V$, $\{L_{x,y},\,R_{x,y},\,T_x,\,M_{x,y},\,U_x\}$ generates total inner mapping groups in $\V$, $\{A_{x,y}^\cdot,\,B_{x,y}^\cdot\}$ generates inner mapping groups in $\V$, and $\{A_{x,y}^\cdot,\,B_{x,y}^\cdot,\,A_{x,y}^\ldiv\}$ generates total inner mapping groups in $\V$.
\item[$\bullet$]
    Let $\V$ be a variety of commutative loops. Then $\{L_{x,y}\}$ generates inner mapping groups in $\V$, $\{L_{x,y},\,M_{x,y},\,U_x\}$ generates total inner mapping groups in $\V$, and $\{A_{x,y}^\cdot,\,A_{x,y}^\ldiv\}$ generates total inner mapping groups in $\V$.
\item[$\bullet$]
    Let $\V$ be a variety of inverse property loops. Then $\{L_{x,y},\,T_x,\,J\}$ generates total inner mapping groups in $\V$.
\item[$\bullet$]
    Let $\V$ be a variety of groups. Then $\{T_x\}$ generates inner mapping groups in $\V$, and $\{T_x,\,J\}$ generates total inner mapping groups in $\V$.
\end{enumerate}
\end{example}

The following lemma explains how (total) inner mapping groups can be generated uniformly. We will write $\W^{\pm1}$ for $\{W,\,W^{-1};\;W\in\W\}$, where $W^{-1}$ is the word obtained by formally inverting $W$.

\begin{lemma}\label{Lm:free}
Let $\V$ be a variety of loops, $\W$ a set generating (total) inner mapping groups in $\V$ and $V$ an inner word for $\V$. Then there exist $W^i\in\W^{\pm1}$ and terms $t_j^i$ such that
\begin{displaymath}
    V_{\bar a}=W^1_{t_1^1(\bar a),\dots,t_{r_1}^1(\bar a)}\dots W^k_{t_1^k(\bar a),\dots,t_{r_k}^k(\bar a)}
\end{displaymath}
for every $Q\in\V$ and every $a_1,\dots,a_n\in Q$.
\end{lemma}

\begin{proof}
We will write down the case of inner mapping groups. For total inner mapping groups, replace every reference to $\inn Q$ by $\totinn Q$.

Let $F$ be the free loop in $\V$ on $n$ generators $x_1,\dots,x_n$. We know that $V_{\bar x}\in\inn F$ (as it does for any $Q\in\V$ and every choice of parameters from $Q$). Since $\inn F$ is generated by all $W_{\bar t}$ such that $W\in\W$ and $t_1,\dots,t_n\in F$ (as so does $\inn Q$ for every $Q\in\V$), there exist words $W^1,\dots,W^k\in\W^{\pm1}$ and parameters $t_j^i\in F$ such that
\begin{displaymath}
    V_{\bar x}=W^1_{t_1^1,\dots,t_{r_1}^1}\dots W^k_{t_1^k,\dots,t_{r_k}^k}.
\end{displaymath}
Notice that elements of $F$ are just terms in variables $x_1,\dots,x_k$, and inner mappings in $F$ can be considered as inner words for $\V$, since the equality $V(1)=1$ in $F$ becomes an identity true in $\V$.

The whole situation easily maps into every loop in $\V$. Given $Q\in\V$ and $a_1\dots,a_n\in Q$, let $f:F\to Q$ be the homomorphism induced by $x_1\mapsto a_1$, \dots, $x_n\mapsto a_n$. Upon applying the homomorphism, we obtain the equality in the statement of the lemma.
\end{proof}

Assuming the notation of Lemma \ref{Lm:free}, we say that, in the variety $\V$, the mappings induced by $V$ are \emph{uniformly generated} using the words $W^i$ and the terms $t_j^i$.
Since the statement of Lemma \ref{Lm:free} is a bit technical, we again make it clear with an example:

\begin{example}
Let $Q$ be an arbitrary loop, $a\in Q$, and consider the mapping $f_a$ defined by $f_a(z)=(z\ldiv a)\ldiv a$. Then $f_a\in\totmlt Q$, because $f_a=M_aM_a$. It is an inner mapping because $f_a(1)=1$. Consequently, Proposition \ref{Pr:TotInnGens} says that $f_a$ is a product of mappings $A_{x,y}^{\cdot}$, $B_{x,y}^\cdot$, $A_{x,y}^\ldiv$ and their inverses, for some choice of parameters $x$, $y\in Q$.

But what if we choose a different $b\in Q$, or if we work in a different loop? Do we get an analogous generating word for $f_b$? Lemma \ref{Lm:free} guarantees that there is a uniform way of generating $f_a$ in every loop for every choice of $a$, because $f_a$ is induced by the inner word $M_xM_x$.

Note, however, that the proof of Lemma \ref{Lm:free} is not constructive, so it is not at all clear how to generate a particular inner mapping from a set of words that generates (total) inner mapping groups.
\end{example}

\section{Proof of the fundamental theorem}\label{Sc:Proof}

\subsection{Auxiliary lemmas}

The principal difficulty with the proof of Theorem \ref{Th:CongrComInf} is the fact that to establish centrality, one has to consider the term condition $\TC$ for all terms. Lemma \ref{Lm:TC} below reduces the set of terms that need to be considered.

A term $t(x_1,\dots,x_n)$ is called \emph{slim with respect to $x_1$}, if there is only one occurrence of $x_1$ in $t$, if this occurrence is in the lowest level of $t$, and if every node of $t$ has at most one branch of length greater than $1$.

For example, the term on the left is slim with respect to $x_1$ but not with respect to the other variables, and the term on the right is not slim with respect to any variables:
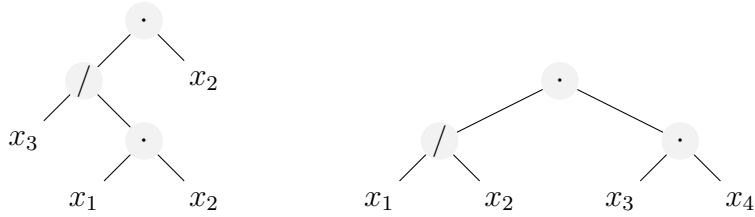
\begin{displaymath}
    \begin{tikzpicture}
    [scale = 0.8, operation/.style={circle,fill=gray!10,inner sep=0pt,minimum size=5mm}, element/.style={}]
        \node (21) at (10,0) [operation] {$\cdot$};
        \node (22) at (9,-1) [operation] {$\rdiv$} edge (21);
        \node (23) at (11,-1) [element] {$x_2$} edge (21);
        \node (24) at (8,-2) [element] {$x_3$} edge (22);
        \node (25) at (10,-2) [operation] {$\cdot$} edge (22);
        \node (26) at (9,-3) [element] {$x_1$} edge (25);
        \node (27) at (11,-3) [element] {$x_2$} edge (25);
    \end{tikzpicture}
    \quad\quad\quad\quad
    \begin{tikzpicture}
    [scale = 0.8, operation/.style={circle,fill=gray!10,inner sep=0pt,minimum size=5mm}, element/.style={}]
        \node (21) at (10,0) [operation] {$\cdot$};
        \node (22) at (8,-1) [operation] {$\rdiv$} edge (21);
        \node (23) at (12,-1) [operation] {$\cdot$} edge (21);
        \node (24) at (7,-2) [element] {$x_1$} edge (22);
        \node (25) at (9,-2) [element] {$x_2$} edge (22);
        \node (26) at (11,-2) [element] {$x_3$} edge (23);
        \node (27) at (13,-2) [element] {$x_4$} edge (23);
    \end{tikzpicture}
\end{displaymath}

\begin{lemma}\label{Lm:TC}
Let $\A$ be an algebra and $\alpha,\beta,\delta$ its congruences. If $\TC(t,\alpha,\beta,\delta)$ holds for every term $t$ that is slim with respect to the first variable, then $C(\alpha,\beta;\delta)$ holds.
\end{lemma}

\begin{proof}
We prove the lemma in two steps. First, we show that $\TC(t,\alpha,\beta,\delta)$ holds for every term $t$, not necessarily slim, with only one occurrence of the first variable. Let $t=t(x_1,\dots,x_{n+1})$ be such a term. Define a new term, $s(x_1,y_1,\dots,y_k)$, by replacing each maximal subterm $s_i(x_2,\dots,x_n)$ of $t$ not containing $x_1$ by a new variable $y_i$. Then $s$ is slim with respect to $x_1$, and $t(x_1,\dots,x_{n+1})=s(x_1,s_1,\dots,s_k)$, as illustrated below:
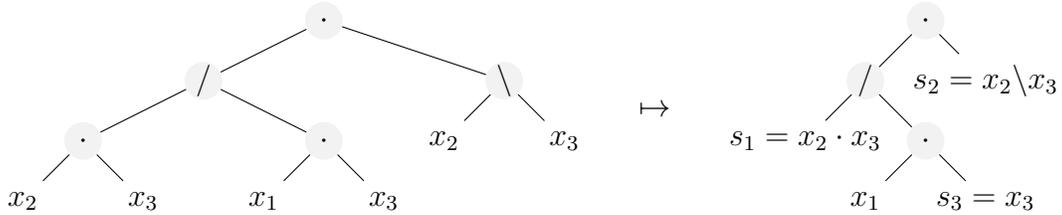
\begin{displaymath}
    \begin{tikzpicture}
    [scale = 0.8, operation/.style={circle,fill=gray!10,inner sep=0pt,minimum size=5mm}, element/.style={}]
        \node (1) at (0,0) [operation] {$\cdot$};
        \node (2) at (-2,-1) [operation] {$\rdiv$} edge (1);
        \node (3) at (3,-1) [operation] {$\ldiv$} edge (1);
        \node (4) at (-4,-2) [operation] {$\cdot$} edge (2);
        \node (5) at (0,-2) [operation] {$\cdot$} edge (2);
        \node (6) at (-5,-3) [element] {$x_2$} edge (4);
        \node (7) at (-3,-3) [element] {$x_3$} edge (4);
        \node (8) at (-1,-3) [element] {$x_1$} edge (5);
        \node (9) at (1,-3) [element] {$x_3$} edge (5);
        \node (10) at (2,-2) [element] {$x_2$} edge (3);
        \node (11) at (4,-2) [element] {$x_3$} edge (3);
        \node (15) at (5.5,-1.5) [element] {$\mapsto$};
        \node (21) at (10,0) [operation] {$\cdot$};
        \node (22) at (9,-1) [operation] {$\rdiv$} edge (21);
        \node (23) at (11,-1) [element] {$s_2=x_2\ldiv x_3$} edge (21);
        \node (24) at (8,-2) [element] {$s_1=x_2\cdot x_3$} edge (22);
        \node (25) at (10,-2) [operation] {$\cdot$} edge (22);
        \node (26) at (9,-3) [element] {$x_1$} edge (25);
        \node (27) at (11,-3) [element] {$s_3=x_3$} edge (25);
    \end{tikzpicture}
\end{displaymath}
Let $a\,\alpha\,b$, $\bar u\,\beta\,\bar v$. Then indeed $s_i(u_1,\dots,u_n)\,\beta\,s_i(v_1,\dots,v_n)$ for every $i$, because $\beta$ is a congruence. Suppose that $t(a,u_1,\dots,u_n)\,\delta\,t(a,v_1,\dots,v_n)$, which we can restate as
\begin{displaymath}
    s(a,s_1(u_1,\dots,u_n),\dots,s_k(u_1,\dots,u_n))\,\delta\,s(a,s_1(v_1,\dots,v_n),\dots,s_k(v_1,\dots,v_n)).
\end{displaymath}
Using $\TC(s,\alpha,\beta,\delta)$, we deduce
\begin{displaymath}
    s(b,s_1(u_1,\dots,u_n),\dots,s_k(u_1,\dots,u_n))\,\delta\,s(b,s_1(v_1,\dots,v_n),\dots,s_k(v_1,\dots,v_n)),
\end{displaymath}
which means $t(b,u_1,\dots,u_n)\,\delta\,t(b,v_1,\dots,v_n)$. Hence $\TC(t,\alpha,\beta,\delta)$ holds for every term $t$ with a single occurrence of the first variable.

In the second step, consider a general term $t=t(x_1,\dots,x_{n+1})$ with $k$ occurrences of $x_1$. Define a new term, $s(y_1,\dots,y_k,x_2,\dots,x_{n+1})$, by replacing every occurrence of $x_1$ by a unique new variable $y_1$, $\dots$, $y_k$. We will use $\TC(s,\alpha,\beta,\delta)$ repeatedly, once for each of the variables $y_1$, $\dots$, $y_k$, starting with  $t(a,u_1,\dots,u_n)\,\delta\,t(a,v_1,\dots,v_n)$, i.e., with
\begin{displaymath}
    s(a,\dots,a,u_1,\dots,u_n)\,\delta\,s(a,\dots,a,v_1,\dots,v_n).
\end{displaymath}
Then
\begin{align*}
    s(b,a,\dots,a,u_1,\dots,u_n)\,&\delta\,s(b,a,\dots,a,v_1,\dots,v_n),\\
    s(b,b,a,\dots,a,u_1,\dots,u_n)\,&\delta\,s(b,b,a,\dots,a,v_1,\dots,v_n),\\
    &\vdots\\
    s(b,\dots,b,u_1,\dots,u_n)\,&\delta\,s(b,\dots,b,v_1,\dots,v_n),
\end{align*}
which translates into $t(b,u_1,\dots,u_n)\,\delta\,t(b,v_1,\dots,v_n)$, and we are through.
\end{proof}

The following lemma is more general than we need in the proof of Theorem \ref{Th:Fundamental}, but it is readily available using Lemma \ref{Lm:free}.

\begin{lemma}\label{Lm:CongrComm-generators}
Let $\V$ be a variety of loops, $\W$ a set of words that generates total inner mapping groups in $\V$, $Q\in\V$ and $\alpha$, $\beta$ congruences of $Q$. Put
\begin{displaymath}
    \delta = \Cg{ (W_{\bar u}(a), W_{\bar v}(a) );\;W\in\mathcal W,\,1\,\alpha\,a,\,\bar u\,\beta\,\bar v }.
\end{displaymath}
Then $(V_{\bar u}(a),V_{\bar v}(a))\in\delta$ for every word $V$ that is inner for $\V$ and for every $1\,\alpha\,a$, $\bar u\,\beta\,\bar v$.
\end{lemma}

\begin{proof}
We first note that $W^{-1}_{\bar u}(a)\,\delta\,W^{-1}_{\bar v}(a)$ for every $W\in\mathcal W$, $1\,\alpha\,a$ and $\bar u\,\beta\,\bar v$. Indeed, since $1=W^{-1}_{\bar u}(1)\,\alpha\,W^{-1}_{\bar u}(a)$, we get $a = W_{\bar u}(W^{-1}_{\bar u}(a))\,\delta\,W_{\bar v}(W^{-1}_{\bar u}(a))$, and thus $W^{-1}_{\bar u}(a)\,\delta\,W^{-1}_{\bar v}(a)$.

According to Lemma \ref{Lm:free}, there exist $W^1,\dots,W^k\in\W^{\pm1}$ and terms $t_i^j$ such that they uniformly generate the mappings induced by $V$.
To keep our notation simple, we will use the shorthand $W^i_{\bar t(\bar z)}$ for the mapping $W^i_{t_1^i(\bar z),\dots,t_{r_i}^i(\bar z)}$.
Notice that if $x\,\delta\,y$, then $W^i_{\bar t(\bar z)}(x)\,\delta\,W^i_{\bar t(\bar z)}(y)$ for every $\bar z\in Q$.

Fix $a\,\alpha\,1$. An easy induction shows the result: for $i=k$, we have $W^k_{\bar t(\bar u)}(a)\,\delta\,W^k_{\bar t(\bar v)}(a)$ by the definition of $\delta$, and if
$W^{i+1}_{\bar t(\bar u)}\dots W^k_{\bar t(\bar u)}(a)\,\delta\,W^{i+1}_{\bar t(\bar v)}\dots W^k_{\bar t(\bar v)}(a)$, then
$$W^{i}_{\bar t(\bar u)}W^{i+1}_{\bar t(\bar u)}\dots W^k_{\bar t(\bar u)}(a)\,\delta\,W^i_{\bar t(\bar u)}W^{i+1}_{\bar t(\bar v)}\dots W^k_{\bar t(\bar v)}(a)\,\delta\,W^i_{\bar t(\bar v)}W^{i+1}_{\bar t(\bar v)}\dots W^k_{\bar t(\bar v)}(a),$$
where the former equivalence follows from the induction assumption, and the latter one follows from the definition of $\delta$, because $W^{i+1}_{\bar t(\bar v)}\dots W^k_{\bar t(\bar v)}$ is an inner mapping, and thus $1=W^{i+1}_{\bar t(\bar v)}\dots W^k_{\bar t(\bar v)}(1)\,\alpha\,W^{i+1}_{\bar t(\bar v)}\dots W^k_{\bar t(\bar v)}(a)$.
\end{proof}

\subsection{Proof of the fundamental theorem}

We are ready to prove Theorem \ref{Th:CongrComInf}:

\begin{proof}[Proof of Theorem \ref{Th:CongrComInf}]
Let $\delta = \Cg{ (W_{\bar u}(a), W_{\bar v}(a) );\;W\in\mathcal W,\,1\,\alpha\,a,\,\bar u\,\beta\,\bar v }$ and $A=N_\alpha$.

First we show $[\alpha,\beta]\supseteq\delta$. We only need to check that $W_{\bar u}(a)\,[\alpha,\beta]\,W_{\bar v}(a)$ whenever $W$ is inner, $a\in A$ and $\bar u\,\beta\,\bar v$. Note that $W_{\bar u}(1)= 1 = W_{\bar v}(1)$ because $W$ is an inner word. In particular $W_{\bar u}(1)\,[\alpha,\beta]\,W_{\bar v}(1)$ and the term condition $C(\alpha,\beta;[\alpha,\beta])$ for $t(\bar x,y)=W_{\bar x}(y)$ gives $W_{\bar u}(a)\,[\alpha,\beta]\,W_{\bar v}(a)$.

Now we show $[\alpha,\beta]\subseteq\delta$. Recall that $[\alpha,\beta]$ is the smallest congruence with $C(\alpha,\beta;[\alpha,\beta])$. In order to show $[\alpha,\beta]\subseteq\delta$, it is sufficient to check that $C(\alpha,\beta;\delta)$. We will write $x\equiv y$ if $x\,\delta\,y$.

According to Lemma \ref{Lm:TC}, we need to check the term condition $\TC(t,\alpha,\beta,\delta)$ for every term $t$ that is slim in the first coordinate. Consider such a term $t$, $a\,\alpha\,b$, $\bar u\,\beta\,\bar v$ and assume $t(a,u_1,\dots,u_n)\equiv t(a,v_1,\dots,v_n)$. We will show that $t(b,u_1,\dots,u_n)\equiv t(b,v_1,\dots,v_n)$.

Let $d$ be the depth of the (unique) occurrence of $x_1$ in $t$. We will construct a sequence of $(n+2)$-ary terms $s_0,\dots,s_d$ and elements $a_0,\dots,a_d\in A$ and $a_0',\dots,a_d'\in A$ satisfying the following conditions for every $i=0,\dots,d$:
\begin{enumerate}
	\item $s_i$ has a single occurrence of $x_1$ at depth $d+1-i$, and it appears in a subterm of the form $x_1\cdot s$ for some term $s$,
	\item $s_i(1,x_1,\dots,x_{n+1})=t(x_1,\dots,x_{n+1})$,
	\item $s_i(a_i,b,u_1,\dots,u_n)\equiv s_i(a_i',b,v_1,\dots,v_n)$,
	\item $a_i\equiv a_i'$ and $a_i,a_i'\in A$.
\end{enumerate}

Assuming existence of the sequences, it is easy to finish the proof of the theorem: it follows from (1) and (2) that $s_d(x_1,\dots,x_{n+2})=x_1\cdot t(x_2,\dots,x_{n+2})$, and using (3),
\begin{displaymath}
    a_d\cdot t(b,u_1,\dots,u_n)=s_d(a_d,b,u_1,\dots,u_n)\equiv s_d(a_d',b,v_1,\dots,v_n)=a_d'\cdot t(b,v_1,\dots,v_n).
\end{displaymath}
According to (4), $a_d\equiv a_d'$, so we can cancel in $Q/\delta$ and get $t(b,u_1,\dots,u_n)\equiv t(b,v_1,\dots,v_n).$

Now we will construct the sequences. In the initial step, take
\begin{displaymath}
    s_0(x_1,\dots,x_{n+2})=t(x_1\cdot x_2, x_3,\dots, x_{n+2}),
\end{displaymath}
and let $a_0=a_0'=a/b$. It is readily seen that the conditions (1), (2), (4) are satisfied, and (3) follows from the assumption that $t(a,u_1,\dots,u_n)\equiv t(a,v_1,\dots,v_n)$.

Induction step: given $s_i,a_i,a_i'$ satisfying the conditions, we will construct $s_{i+1},a_{i+1},a_{i+1}'$.
Since $t$ is slim, one of the following configurations takes place near the occurrence of $x_1$ in $s_i$, for some variable $x_k$, some term $s(x_2,\dots,x_{n+2})$ and some operation $\circ\in\{\cdot,\ldiv,\rdiv\}$:
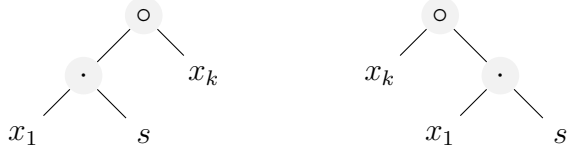
\begin{displaymath}
    \treeL{$x_1$}{$\cdot$}{$s$}{$\circ$}{$x_k$}\qquad\qquad \treeR{$x_k$}{$\circ$}{$x_1$}{$\cdot$}{$s$}
\end{displaymath}
Let $W_{x,y}=A_{x,y}^\circ$ in the former case and $W_{x,y}=B_{x,y}^\circ$ in the latter case. (Recall that $(z\cdot x)\circ y=A_{x,y}^\circ(z)\cdot(x\circ y)$ and $y\circ(z\cdot x)=B_{x,y}^\circ(z)\cdot(y\circ x)$.) Define
\begin{displaymath}
    a_{i+1}=W_{s(b,u_1,\dots,u_n),u_{k-2}}(a_i), \qquad a_{i+1}'=W_{s(b,v_1,\dots,v_n),v_{k-2}}(a_i'),
\end{displaymath}
and let $s_{i+1}$ result from $s_i$ by replacing the subterm $(x_1\cdot s)\circ x_k$ by $x_1\cdot(s\circ x_k)$, or
$x_k\circ(x_1\cdot s)$ by $x_1\cdot(x_k\circ s)$, respectively. Graphically, the configuration near the occurrence of $x_1$ in $s_{i+1}$ becomes one of the following:
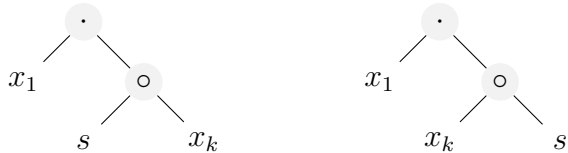
\begin{displaymath}
    \treeR{$x_1$}{$\cdot$}{$s$}{$\circ$}{$x_k$}\qquad\qquad \treeR{$x_1$}{$\cdot$}{$x_k$}{$\circ$}{$s$}
\end{displaymath}
It is readily seen that the conditions (1), (2) hold for $s_{i+1}$ too.

Let us verify condition (3). Suppose that we are in the former case. Then near $x_1$ the evaluated term $s_i(a_i,b,u_1,\dots,u_n)$ gives $(a_i\cdot s(b,u_1,\dots,u_n))\circ u_{k-2}$, which is equal to $A^\circ_{s(b,u_1,\dots,u_n),u_{k-2}}(a_i)\cdot (s(b,u_1,\dots,u_n)\circ u_{k-2}) = a_{i+1}\cdot (s(b,u_1,\dots,u_n)\circ u_{k-2})$, which is how the evaluated term $s_{i+1}(a_{i+1},b,u_1,\dots,u_n)$ looks near $x_1$. In other words, $s_i(a_i,b,u_1,\dots,u_n)=s_{i+1}(a_{i+1},b,u_1,\dots,u_n)$. Similarly, $s_i(a_i',b,v_1,\dots,v_n)=s_{i+1}(a_{i+1}',b,v_1,\dots,v_n)$. Then we can use (3) for $s_i$ to conclude that
\begin{displaymath}
    s_{i+1}(a_{i+1},b,u_1,\dots,u_n) = s_i(a_i,b,u_1,\dots,u_n) \equiv s_i(a_i',b,v_1,\dots,v_n)=s_{i+1}(a_{i+1}',b,v_1,\dots,v_n),
\end{displaymath}
so that (3) holds for $s_{i+1}$. Similarly in the latter case.

To check condition (4), observe that $a_{i+1},a_{i+1}'\in A$ because the normal subloop $A$ is closed under inner mappings by Proposition \ref{Pr:TotInn}, and that
\begin{displaymath}
    a_{i+1}=W_{s(b,u_1,\dots,u_n),u_{k-2}}(a_i)\equiv W_{s(b,u_1,\dots,u_n),u_{k-2}}(a_i')\equiv W_{s(b,v_1,\dots,v_n),v_{k-2}}(a_i')=a_{i+1}',
\end{displaymath}
where the former equivalence follows from the fact that $a_i\equiv a_i'$ and $\equiv$ is a congruence, and the latter one follows from Lemma \ref{Lm:CongrComm-generators}.
\end{proof}

The following example illustrates the inductive algorithm used in the proof of Theorem \ref{Th:CongrComInf}.

\begin{example}
Let $t(x_1,x_2,x_3,x_4)=x_3\rdiv((x_1\ldiv x_4)\cdot x_2)$. The sequence $s_0,s_1,s_2,s_3$ is depicted below. Notice the occurrence of $x_1$ ``climbing the tree'' to the top level, while the structure of the rest of the tree remains intact. Also notice that $s_3(x_1,\dots,x_5)=x_1\cdot t(x_2,\dots,x_5)$.
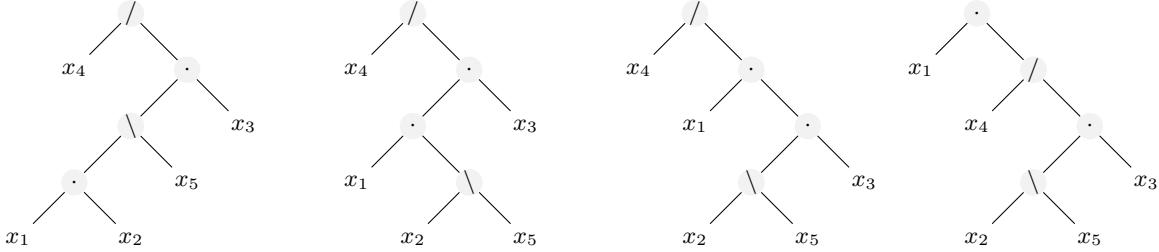
\begin{scriptsize}
\begin{displaymath}
    \begin{tikzpicture}
    [scale = 0.75, operation/.style={circle,fill=gray!10,inner sep=0pt,minimum size=3.5mm}, element/.style={}]
        \node (1) at (0,0) [operation] {$\rdiv$};
        \node (2) at (-1,-1) [element] {$x_4$} edge (1);
        \node (3) at (1,-1) [operation] {$\cdot$} edge (1);
        \node (4) at (0,-2) [operation] {$\ldiv$} edge (3);
        \node (5) at (2,-2) [element] {$x_3$} edge (3);
        \node (6) at (-1,-3) [operation] {$\cdot$} edge (4);
        \node (7) at (1,-3) [element] {$x_5$} edge (4);
        \node (8) at (-2,-4) [element] {$x_1$} edge (6);
        \node (9) at (0,-4) [element] {$x_2$} edge (6);
        \node (11) at (5,0) [operation] {$\rdiv$};
        \node (12) at (4,-1) [element] {$x_4$} edge (11);
        \node (13) at (6,-1) [operation] {$\cdot$} edge (11);
        \node (14) at (5,-2) [operation] {$\cdot$} edge (13);
        \node (15) at (7,-2) [element] {$x_3$} edge (13);
        \node (16) at (4,-3) [element] {$x_1$} edge (14);
        \node (17) at (6,-3) [operation] {$\ldiv$} edge (14);
        \node (18) at (5,-4) [element] {$x_2$} edge (17);
        \node (19) at (7,-4) [element] {$x_5$} edge (17);
        \node (21) at (10,0) [operation] {$\rdiv$};
        \node (22) at (9,-1) [element] {$x_4$} edge (21);
        \node (23) at (11,-1) [operation] {$\cdot$} edge (21);
        \node (24) at (10,-2) [element] {$x_1$} edge (23);
        \node (25) at (12,-2) [operation] {$\cdot$} edge (23);
        \node (26) at (11,-3) [operation] {$\ldiv$} edge (25);
        \node (27) at (13,-3) [element] {$x_3$} edge (25);
        \node (28) at (10,-4) [element] {$x_2$} edge (26);
        \node (29) at (12,-4) [element] {$x_5$} edge (26);
        \node (51) at (15,0) [operation] {$\cdot$};
        \node (52) at (14,-1) [element] {$x_1$} edge (51);
        \node (53) at (16,-1) [operation] {$\rdiv$} edge (51);
        \node (54) at (15,-2) [element] {$x_4$} edge (53);
        \node (55) at (17,-2) [operation] {$\cdot$} edge (53);
        \node (56) at (16,-3) [operation] {$\ldiv$} edge (55);
        \node (57) at (18,-3) [element] {$x_3$} edge (55);
        \node (58) at (15,-4) [element] {$x_2$} edge (56);
        \node (59) at (17,-4) [element] {$x_5$} edge (56);
    \end{tikzpicture}
\end{displaymath}
\end{scriptsize}
\end{example}

\subsection{The fundamental theorem in loops with a finiteness condition}

Assume that $Q$ has a bound on the order of all left and right translations, i.e., there is an integer $n>0$ such that $L_x^n=R_x^n=1$ for every $x\in Q$. (This is certainly true in every finite loop.) Then the theory developed so far simplifies considerably, because we can avoid the division operations: we have
\begin{displaymath}
    x\ldiv y=L_x^{-1}(y)=L_x^{n-1}(y)=\underbrace{x(\dots(x(x}_{n-1}y))\dots),
\end{displaymath}
and dually for the right division. Hence, for every loop term $t$, there is a \emph{multiplicative} term $s$ (i.e., a term in the language of multiplication) such that $t=s$ is a valid identity in every loop where $L_x^n=R_x^n=1$ for every $x$. So, in the term condition $\TC(t,\alpha,\beta,\delta)$ verified in the proof of Theorem \ref{Th:CongrComInf}, we only need to consider \emph{multiplicative slim terms}. Hence, later in the proof, we only need to use the words $A_{x,y}^\cdot$ and $B_{x,y}^\cdot$, which always induce mappings from $\inn Q$ (while the other options may induce mappings from $\totinn Q$). It means that we only need a weaker version of Lemma \ref{Lm:CongrComm-generators}, for inner words from $\inn Q$, and thus we can choose a weaker $\W$, a set of words that generates inner mapping groups (not necessarily total inner mapping groups) in $\V$. We have proved:

\begin{theorem}\label{Th:FundamentalFinite}
Let $\V$ be a variety of loops and $\mathcal W$ a set of words that generates inner mapping groups in $\V$. If for $Q\in\V$ there is $n>0$ such that $L_x^n=R_x^n=1$ for every $x\in Q$, then
\begin{displaymath}
    [\alpha,\beta] = \Cg{(W_{\bar u}(a), W_{\bar v}(a) );\;W\in\mathcal W,\,1\,\alpha\,a,\,\bar u\,\beta\,\bar v}
\end{displaymath}
for any congruences $\alpha$, $\beta$ of $Q$.
\end{theorem}

\section{Pruning the generating sets}\label{Ss:Pruning}

Assume the notation of Theorem \ref{Th:Fundamental}, that is, let $\V$ be a variety of loops and $\W$ a set of words that generates inner mapping groups in $\V$. For $Q\in\V$ and congruences $\alpha$, $\beta$ of $Q$ call $\Cg{(W_{\bar u}(a), W_{\bar v}(a) );\;W\in\mathcal W,\,1\,\alpha\,a,\,\bar u\,\beta\,\bar v}$ the \emph{congruence induced by} $\W$. In some cases, a word can be removed from $\W$ with no effect on the congruences induced by $\W$. Formally, we say that a word $W$ is \emph{removable} from a set $\W$, if for every loop $Q\in\V$ and every congruences $\alpha$, $\beta$ of $Q$ the congruence induced by $\W$ and the congruence induced by $\W\smallsetminus\{W\}$ are equal.

\begin{proposition}\label{Pr:Pruning}
Let $\V$ be a variety of loops, $\W$ be a set of inner words, $W\in\W$, and $\equiv$ the congruence induced by $\W\smallsetminus\{W\}$.
\begin{enumerate}
	\item[(i)] If $W$ has no parameters, it is removable from $\W$.
	\item[(ii)] The word $W=U_x$ is removable from $\W$, provided that $[a,x,b]\equiv 1$
	for every loop $Q\in\V$, every congruences $\alpha$, $\beta$ of $Q$, and every $1\,\alpha\, a$, $1\,\beta\, b$, $x\in Q$.
	\item[(iii)] The word $W=T_x$ is removable from $\W$, provided that $[a,b]\equiv1$, $[a,b,x]\equiv1$, $[b,a,x]\equiv1$ and $[x,b,a]\equiv1$
	for every loop $Q\in\V$, every congruences $\alpha$, $\beta$ of $Q$, and every $1\,\alpha\, a$, $1\,\beta\, b$, $x\in Q$.
\end{enumerate}
\end{proposition}

\begin{proof}
(i) The diagonal element $(W(a),W(a))$ belongs to every congruence.

(ii) The assumption $[a,x,b]\equiv 1$ can be written as
\begin{equation}\label{Eq:Aux1}
    ax\cdot b \equiv a\cdot xb,\quad\text{ i.e., }\quad R_bR_x(a) \equiv R_{xb}(a).
\end{equation}
Upon replacing $x$ with $a\ldiv x$, and dividing both sides by $a$, we get
\begin{equation}\label{Eq:Aux2}
    a\ldiv (xb) \equiv (a\ldiv x)b,\quad\text{ i.e., }\quad M_{xb}(a) \equiv R_bM_x(a)
\end{equation}
for every $1\,\alpha\, a$, $1\,\beta\, b$, $x\in Q$.
Suppose that $1\,\alpha\,a$ and $u\,\beta\,v$, thus $1=M_v^{-1}R_v(1)\,\alpha\, M_v^{-1}R_v(a)$ and $(v\ldiv u)\,\beta\,1$. Then
\begin{align*}
    U_uU_v^{-1}(a) &= R_u^{-1}M_uM_v^{-1}R_v(a)=R_u^{-1}M_{v(v\ldiv u)}M_v^{-1}R_v(a)&& \\
    &\equiv R_u^{-1}R_{v\ldiv u}M_vM_v^{-1}R_v(a)= R_u^{-1}R_{v\ldiv u}R_v(a) &&\text{by \eqref{Eq:Aux2} with $M_v^{-1}R_v(a)\,\alpha\,1$}\\
    &\equiv R_u^{-1}R_{v(v\ldiv u)}(a)= R_u^{-1}R_u(a)=a &&\text{by \eqref{Eq:Aux1}}.
\end{align*}
Upon replacing $a$ with $U_v(a)$, we obtain $U_u(a)\equiv U_v(a)$.

(iii) The assumptions can be written as
\begin{align*}
ab\equiv ba,&\quad \text{ i.e., }\quad L_b(a)\equiv R_b(a),\\
a\cdot bx\equiv ab\cdot x,&\quad \text{ i.e., }\quad R_{bx}(a) \equiv R_xR_b(a),\\
b\cdot ax\equiv ba\cdot x,&\quad \text{ i.e., }\quad L_bR_{x}(a) \equiv R_xL_b(a),\\
x\cdot ba\equiv xb\cdot a,&\quad \text{ i.e., }\quad L_xL_b(a) \equiv L_{xb}(a),
\end{align*}
for every $1\,\alpha\, a$, $1\,\beta\, b$, $x\in Q$.
Suppose that $1\,\alpha\,a$ and $u\,\beta\,v$, thus $1=L_v^{-1}R_v(1)\,\alpha\, L_v^{-1}R_v(a)$ and $(u/v)\,\beta\,1$. Then
\begin{align*}
    T_uT_v^{-1}(a) &= R_u^{-1}L_uL_v^{-1}R_v(a)=R_u^{-1}L_{(u/v)v}L_v^{-1}R_v(a)&&\\
    &\equiv R_u^{-1}L_{u/v}L_vL_v^{-1}R_v(a)=R_u^{-1}L_{u/v}R_v(a)&&\text{ by $[u/v,v,L_v^{-1}R_v(a)]\equiv 1$}\\
    &\equiv R_u^{-1}R_vL_{u/v}(a)&&\text{ by $[u/v,a,v]\equiv 1$}\\
    &\equiv R_u^{-1}R_vR_{u/v}(a)&&\text{ by $[a,u/v]\equiv 1$}\\
    &\equiv R_u^{-1}R_{(u/v)v}(a)=R_u^{-1}R_u(a)=a &&\text{ by $[a,u/v,v]\equiv 1$}.
\end{align*}
Upon replacing $a$ with $T_v(a)$, we obtain $T_u(a)\equiv T_v(a)$.
\end{proof}

Item (i) of Proposition \ref{Pr:Pruning} applies, for example, to the inverse mapping $J$ in inverse property loops, or more generally, to the mapping $M_1$ in every loop. The assumptions of (ii) are satisfied, for example, if $R_{x,y}\in\W$; then $R_{b,x}(a)\equiv R_{1,x}(a)=1$ whenever $1\,\alpha\, a$, $1\,\beta\, b$ and $x\in Q$, and thus $[a,x,b]\equiv 1$.
Using (iii) is less straightforward; we will do so in Corollary \ref{Cr:FundamentalS}.

The full power of Theorems \ref{Th:CongrComInf} and \ref{Th:FundamentalFinite} is realized only in combination with the results of Section \ref{Sc:TotMlt}, as summarized in Example \ref{Ex:GeneratingSets}.

\begin{corollary}\label{Cr:CongrCom}
Let $Q$ be a loop and $\alpha$, $\beta$ congruences of $Q$. Let $\mathcal W$ be defined as follows:
\begin{enumerate}
\item[(i)] If $Q$ is a loop, let $\mathcal W = \{L_{x,y},R_{x,y},T_x,M_{x,y}\}$ or $\mathcal W=\{A_{x,y}^\cdot,B_{x,y}^\cdot,A_{x,y}^\ldiv\}$.
\item[(ii)] If $Q$ is an inverse property loop, let $\mathcal W = \{L_{x,y},T_x\}$ or $\mathcal W = \{M_{x,y}\}$.
\item[(iii)] If $Q$ is a group, let $\mathcal W = \{T_x\}$.
\item[(iv)] If $Q$ is a commutative loop, let $\mathcal W = \{L_{x,y},M_{x,y}\}$ or $\mathcal W=\{A_{x,y}^\cdot,A_{x,y}^\ldiv\}$.
\end{enumerate}
Then
\begin{displaymath}
    [\alpha,\beta] = \Cg{ (W_{\bar u}(a), W_{\bar v}(a) );\;W\in\mathcal W,\,1\,\alpha\, a,\,\bar{u}\,\beta\,\bar{v} }
\end{displaymath}
for any congruences $\alpha$, $\beta$ of $Q$.
\end{corollary}
\begin{proof}
Apply Theorem \ref{Th:Fundamental} to the generating sets from Example \ref{Ex:GeneratingSets}, and use the pruning principles of Proposition \ref{Pr:Pruning}.
\end{proof}

It follows from Theorem \ref{Th:FundamentalFinite} that for loops with a finiteness condition we can remove $M_{x,y}$ and $A_{x,y}^\ldiv$ from the sets $\W$ in Corollary \ref{Cr:CongrCom} (i,iv).

\begin{problem}
Find additional small generating sets for Corollary \ref{Cr:CongrCom}.
\end{problem}

Proposition \ref{Pr:Pruning} is an ad hoc argument designed specifically to prune the standard generating sets. We therefore ask:

\begin{problem}
Describe systematically when a word is removable from a set of inner words.
\end{problem}

Given a set of words $\W$, we now show a simple trick that decreases the size of the generating set of the commutator, useful for computational purposes. Suppose that some two-parameter inner word $W_{x,y}$ is used in $\mathcal W$. We claim that
\begin{align*}
    \Cg{ (W_{u_1,u_2}(a),&W_{v_1,v_2}(a));\;1\,\alpha\, a,\,\bar u\,\beta\,\bar v } \\
    &=\Cg{ (W_{u,c}(a),W_{v,c}(a)),\,(W_{c,u}(a),W_{c,v}(a));\;1\,\alpha\, a,\,u\,\beta\, v,\,c\in Q }.
\end{align*}
Indeed, choosing $u_2=v_2=c$ (so surely $u_2\,\beta\, v_2$) or $u_1=v_1=c$ shows that the second congruence is a subset of the first. Conversely, if the generators of the second congruence are available then $W_{u_1,u_2}(a)$ is congruent to $W_{u_1,v_2}(a)$, which is in turn congruent to $W_{v_1,v_2}(a)$. Of course, a similar observation holds for inner terms with an arbitrary number of parameters.
Note that the second generating set is generally smaller than the first, but it is more cumbersome to write down.

\section{Elementwise commutators and associators}\label{Sc:AC}

\subsection{Commutators and associators as inner mappings}

We have seen that the words $A_{x,y}^\circ$, $B_{x,y}^\circ$ form a set that generates total inner mapping groups in all loops, and thus can be used to generate congruence commutators. Can these words be replaced by terms that resemble elementwise commutators and associators?

The standard commutator $[x,y]$ defined by $xy = (yx)[x,y]$ will not work, since the terms $t_x(y) = [x,y] = s_y(x)$ do not preserve normality in the variety of loops, in the sense that there exists a loop $Q$ with normal subloop $N$ and $x,y\in Q$, $a\in N$ such that neither of $[x,a]$, $[a,y]$ is in $N$.
The standard associator $[x,y,z]$ is similarly flawed.

With Leong's associator $xy\cdot z = a^L(x,y,z)x\cdot yz$, we can write $A^L_{y,z}(x) = a^L(x,y,z)x$, so that $xy\cdot z = A^L_{y,z}(x)\cdot yz$, and we notice that $A^L_{y,z} = R_{yz}^{-1}R_zR_y$ is an inner mapping. (This is the most important feature of the associator $a^L(x,y,z)$, which seems to have gone unnoticed in \cite{L}.) Consequently, $a^L(N,Q,Q)\subseteq N$ for every $N\unlhd Q$.

We will now define commutators and associators systematically. Let $\V$ be the variety of all loops. A loop term $a(x,y,z)$ is an \emph{associator} if we have $*$, $\circ$, $\circledast$, $\circledcirc\in\{\cdot,\ldiv,\rdiv\}$ such that the following hold:
\begin{enumerate}
    \item[(a)] $(x*y)\circ z = (a(x,y,z)x)\circledast (y\circledcirc z)$ in $\V$ or $z*(y\circ x) = (z\circledast y)\circledcirc (xa(x,y,z))$ in $\V$,
    \item[(b)] $a(1,y,z)=1$ in $\V$ (thus the associators give rise to inner mappings),
    \item[(c)] $a(x,y,z)=1$ on all abelian groups in $\V$.
\end{enumerate}
A loop term $c(x,y)$ is a \emph{commutator} if we have $*$, $\circledast\in\{\cdot,\ldiv,\rdiv\}$ such that the following hold:
\begin{enumerate}
    \item[(a)] $x*y = y\circledast xc(x,y)$ in $\V$ or $y*x = c(x,y)x\circledast y$ in $\V$,
    \item[(b)] $c(1,y)=1$ in $\V$ (thus the commutators give rise to inner mappings),
    \item[(c)] $c(x,y)=1$ on all abelian groups in $\V$.
\end{enumerate}
For instance, given the loop operations $\ast = \cdot$, $\circ = \rdiv$, can we find loop operations $\circledast$, $\circledcirc$ such that
$a^{\cdot\rdiv}(x,y,z)$ defined by $(x\cdot y)\rdiv z = (a^{\cdot\rdiv}(x,y,z)x)\circledast (y\circledcirc z)$ fulfills (b) and (c)? We can equivalently study $A_{y,z}^{\cdot\rdiv}(x)$ defined by $A_{y,z}^{\cdot\rdiv}(x) = a^{\cdot\rdiv}(x,y,z)x$. To make the identity $(x\cdot y)\rdiv z = x\circledast (y\circledcirc z)$ valid in abelian groups, we must choose $(x\cdot y)\rdiv z = x\cdot (y\rdiv z)$ or $(x\cdot y)\rdiv z = x\rdiv (y\ldiv z)$. But only the first choice gives a valid loop identity when $x=1$. Equivalently, only the first choice gives rise to an inner mapping $A_{y,z}^{\cdot\rdiv}$, namely to $A_{y,z}^{\cdot\rdiv} = R_{y\rdiv z}^{-1}R_z^{-1}R_y$. Note that the answer is therefore unique here.

The following lemma can be proved by similar arguments. We omit the straightforward proof.

\begin{lemma}\label{Lm:ACUnique}\
\begin{enumerate}
\item[(i)] Let $\ast$, $\circ$ be loop operations such that $\ast\ne\rdiv$. Then
\begin{itemize}
	\item there are uniquely determined loop operations $\circledast$, $\circledcirc$ such that $(x\ast y)\circ z = x\circledast (y\circledcirc z)$ is an identity in all abelian groups, and such that $(1\ast y)\circ z = 1\circledast (y\circledcirc z)$ is an identity in all loops;
	\item there is a uniquely determined loop operation $\circledast$ such that $x\ast y = y\circledast x$ is an identity in all abelian groups, and such that $1\ast y = y\circledast 1$ is an identity in all loops.
\end{itemize}
\item[(ii)] Let $\ast$, $\circ$ be loop operations such that $\circ\ne\ldiv$. Then
\begin{itemize}
	\item there are uniquely determined loop operations $\circledast$, $\circledcirc$ such that $z \ast (y \circ x) = z \circledast (y\circledcirc x)$ is an identity in all abelian groups, and such that
    $z \ast (y \circ 1) = z \circledast (y\circledcirc 1)$ is an identity in all loops;
	\item there is a uniquely determined loop operation $\circledast$ such that $y\circ x = x\circledast y$ is an identity in all abelian groups, and such that $y\circ 1 = 1\circledast y$ is an identity in all loops.
\end{itemize}
\end{enumerate}
\end{lemma}

The cases excluded in Lemma \ref{Lm:ACUnique} do not have a solution. For instance, with $*=\rdiv$ and $\circ = \cdot$, the only choices of $\circledast$, $\circledcirc\in\{\cdot,\rdiv,\ldiv\}$ that make $(x*y)\circ z = x\circledast(y\circledcirc z)$ valid in all abelian groups are $(x\rdiv y)\cdot z = x\cdot (y\ldiv z)$ and $(x\rdiv y)\cdot z = x\rdiv(y\rdiv z)$, and with $x=1$ these identities become $(1\rdiv y)\cdot z = y\ldiv z$ and $(1\rdiv y)\cdot z = 1\rdiv (y\rdiv z)$. But the first identity is not valid in all loops (take $z=1$ to get $1\rdiv y = y\ldiv 1$), and neither is the second identity (take $y=1$ to get $z=1/(1/z)$, i.e., $1/z = z\ldiv 1$).

We therefore obtain $12$ associators and $4$ commutators, summarized in Table \ref{Tb:AC}. The associators and commutators are presented in the form of inner mappings and also as actual elementwise associators and commutators. The $a$-associators of Table \ref{Tb:AC} are dual to the $b$-associators, and the $c$-commutators are dual to the $d$-commutators, in the order listed in the table. For instance, $a^{\cdot\rdiv}(x,y,z)$ is dual to $b^{\ldiv\cdot}(x,y,z)$.

It is interesting to point out that the multiplicative associators and commutators correspond to the standard generating set for the inner mapping groups:
$$A^{\cdot\cdot}_{y,z}=R_{z,y}, \quad B^{\cdot\cdot}_{y,z}=L_{z,y}, \quad C^{\cdot}_{y}=T_y^{-1}, \quad D^{\cdot}_{y}=T_y.$$

\begin{table}
\begin{displaymath}
\begin{array}{|l|l|l|}
    \hline
    \text{defining identity}&\text{as an inner mapping}&\text{as an associator/commutator}\\
    \hline\hline
    (x\cdot y)\cdot z = A^{\cdot\cdot}_{y,z}(x)\cdot (y\cdot z)& A^{\cdot\cdot}_{y,z} = R_{yz}^{-1}R_zR_y
        & a^{\cdot\cdot}(x,y,z) = A^{\cdot\cdot}_{y,z}(x)/x\\
    (x\cdot y)\ldiv z = A^{\cdot\ldiv}_{y,z}(x)\ldiv (y\ldiv z)& A^{\cdot\ldiv}_{y,z} = M_{y\ldiv z}^{-1}M_zR_y
        & a^{\cdot\ldiv}(x,y,z) = A^{\cdot\ldiv}_{y,z}(x)/x\\
    (x\cdot y)\rdiv z = A^{\cdot\rdiv}_{y,z}(x)\cdot (y\rdiv z)& A^{\cdot\rdiv}_{y,z} = R_{y\rdiv z}^{-1}R_z^{-1}R_y
        & a^{\cdot\rdiv}(x,y,z) = A^{\cdot\rdiv}_{y,z}(x)/x\\
    (x\ldiv y)\cdot z = A^{\ldiv\cdot}_{y,z}(x)\ldiv (y\cdot z)& A^{\ldiv\cdot}_{y,z} = M_{yz}^{-1}R_zM_y
        & a^{\ldiv\cdot}(x,y,z) = A^{\ldiv\cdot}_{y,z}(x)/x\\
    (x\ldiv y)\ldiv z = A^{\ldiv\ldiv}_{y,z}(x)\cdot (y\ldiv z)& A^{\ldiv\ldiv}_{y,z} = R_{y\ldiv z}^{-1}M_zM_y
        & a^{\ldiv\ldiv}(x,y,z) = A^{\ldiv\ldiv}_{y,z}(x)/x\\
    (x\ldiv y)\rdiv z = A^{\ldiv\rdiv}_{y,z}(x)\ldiv (y\rdiv z)& A^{\ldiv\rdiv}_{y,z} = M_{y\rdiv z}^{-1}R_z^{-1}M_y
        & a^{\ldiv\rdiv}(x,y,z) = A^{\ldiv\rdiv}_{y,z}(x)/x\\
    \hline
    z\cdot (y\cdot x) = (z\cdot y)\cdot B^{\cdot\cdot}_{y,z}(x)& B^{\cdot\cdot}_{y,z} = L_{zy}^{-1}L_zL_y
        & b^{\cdot\cdot}(x,y,z) = x\ldiv B^{\cdot\cdot}_{y,z}(x)\\
    z\rdiv (y\cdot x) = (z\rdiv y)\rdiv B^{\rdiv\cdot}_{y,z}(x)& B^{\rdiv\cdot}_{y,z} = M_{z\rdiv y}^{-1}M_z^{-1}L_y
        & b^{\rdiv\cdot}(x,y,z) = x\ldiv B^{\rdiv\cdot}_{y,z}(x)\\
    z\ldiv (y\cdot x) = (z\ldiv y)\cdot B^{\ldiv\cdot}_{y,z}(x)& B^{\ldiv\cdot}_{y,z} = L_{z\ldiv y}^{-1}L_z^{-1}L_y
        & b^{\ldiv\cdot}(x,y,z) = x\ldiv B^{\ldiv\cdot}_{y,z}(x)\\
    z\cdot (y\rdiv x) = (z\cdot y)\rdiv B^{\cdot\rdiv}_{y,z}(x)& B^{\cdot\rdiv}_{y,z} = M_{z\cdot y}L_zM_y^{-1}
        & b^{\cdot\rdiv}(x,y,z) = x\ldiv B^{\cdot\rdiv}_{y,z}(x)\\
    z\rdiv (y\rdiv x) = (z\rdiv y)\cdot B^{\rdiv\rdiv}_{y,z}(x)& B^{\rdiv\rdiv}_{y,z} = L_{z\rdiv y}^{-1}M_z^{-1}M_y^{-1}
        & b^{\rdiv\rdiv}(x,y,z) = x\ldiv B^{\rdiv\rdiv}_{y,z}(x)\\
    z\ldiv (y\rdiv x) = (z\ldiv y)\rdiv B^{\ldiv\rdiv}_{y,z}(x)& B^{\ldiv\rdiv}_{y,z} = M_{z\ldiv y}L_z^{-1}M_y^{-1}
        & b^{\ldiv\rdiv}(x,y,z) = x\ldiv B^{\ldiv\rdiv}_{y,z}(x)\\
    \hline
    x\cdot y = y\cdot C^{\cdot}_y(x)& C^{\cdot}_y = L_y^{-1}R_y
        & c^{\cdot}(x,y) = x\ldiv C^{\cdot}_y(x)\\
    x\ldiv y = y\rdiv C^{\ldiv}_y(x)& C^{\ldiv}_y = M_yM_y
        & c^{\ldiv}(x,y) = x\ldiv C^{\ldiv}_y(x)\\
    \hline
    y\cdot x = D^{\cdot}_y(x)\cdot y& D^{\cdot}_y = R_y^{-1}L_y
        & d^{\cdot}(x,y) = D^{\cdot}_y(x)/x\\
    y\rdiv x = d^{\rdiv}_y(x)\ldiv y& D^{\rdiv}_y = M_y^{-1}M_y^{-1}
        & d^{\rdiv}(x,y) = D^{\rdiv}_y(x)/x\\
    \hline
\end{array}
\end{displaymath}
\caption{Commutators and associators that yield inner mappings in loops.}
\label{Tb:AC}
\end{table}

The following discussion will be important with respect to the choice of associators and commutors that generate the derived subloops and the associator subloops.

\begin{lemma}\label{Lm:Factors}
Let $Q$ be a loop and $N\unlhd Q$.
\begin{enumerate}
\item[(i)] Let $a$ denote one of the associators $a^{\cdot\cdot}$, $a^{\cdot\rdiv}$, $a^{\ldiv\cdot}$, $a^{\ldiv\rdiv}$, $b^{\cdot\cdot}$, $b^{\ldiv\cdot}$, $b^{\cdot\rdiv}$, $b^{\ldiv\rdiv}$. Then $Q/N$ is a group if and only if $\{a(x,y,z);\;x,y,z\in Q\}\subseteq N$.
\item[(ii)] Let $a$ denote one of the associators $a^{\cdot\ldiv}$, $a^{\ldiv\ldiv}$, $b^{\rdiv\cdot}$, $b^{\rdiv\rdiv}$. Then $Q/N$ is an abelian group if and only if $\{a(x,y,z);\;x,y,z\in Q\}\subseteq N$.
\end{enumerate}
\end{lemma}
\begin{proof}
We will give the proof for two cases and leave the remaining ten to the reader.

Let $a=a^{\cdot\cdot}$. Suppose that $a(x,y,z)\in N$ for every $x$, $y$, $z\in N$. In $Q/N$, $a(x,y,z)=1$, hence $(xy)z = (a(x,y,z)x)(yz) = x(yz)$. Conversely, if $Q/N$ is a group then, in $Q/N$, $x(yz) = (xy)z = (a(x,y,z)x)(yz)$, and $a(x,y,z)=1$ (or $a(x,y,z)\in N$) follows by cancelation in $Q/N$.

Let $a=a^{\cdot\ldiv}$. Suppose that $a(x,y,z)\in N$ for every $x$, $y$, $z\in N$. In $Q/N$, $a(x,y,z)=1$, hence $(xy)\ldiv z = (a(x,y,z)x)\ldiv(y\ldiv z) = x\ldiv(y\ldiv z)$, so $y(x((xy)\ldiv z)) = z$. Substituting $z=xy$, we obtain commutativity. Substituting $z=xy\cdot u$ and using commutativity, we obtain associativity. Conversely, if $Q/N$ is an abelian group then, in $Q/N$, $x^{-1}(y^{-1}z) = (xy)^{-1}z = (xy)\ldiv z = (a(x,y,z)x)\ldiv (y\ldiv z) = (a(x,y,z)x)^{-1}(y^{-1}z)$, and $a(x,y,z)=1$ (or $a(x,y,z)\in N$) follows by cancelation in $Q/N$.
\end{proof}

\subsection{The fundamental theorem in terms of commutators and associators}

The machinery of Theorem \ref{Th:CongrComInf} can now be applied to various subsets of the commutators and associators in Table \ref{Tb:AC}.

Whether we work with the inner mappings or with the elementwise commutators and associators is irrelevant. Indeed, for trivial reasons, the two elements $A^{\cdot\cdot}_{y_1,z_1}(x)$, $A^{\cdot\cdot}_{y_2,z_2}(x)$ are congruent if and only if the two elements $ a^{\cdot\cdot}(x,y_1,z_1) = A^{\cdot\cdot}_{y_1,z_1}(x)/x$, $a^{\cdot\cdot}(x,y_2,z_2) = A^{\cdot\cdot}_{y_2,z_2}(x)/x $ are congruent; and similarly for all other commutators/associators.\footnote{However, inner mappings can be composed to form a group while commutators and associators cannot be composed. This is why Theorem \ref{Th:CongrComInf} is stated in terms of inner mappings.}

Using the operations from Table \ref{Tb:AC}, we can reformulate Corollary \ref{Cr:CongrCom} as follows:

\begin{corollary}\label{Cr:CongrComAC}
Let $Q$ be a loop and $\alpha$, $\beta$ congruences of $Q$. Let $\mathcal W$ be defined as follows:
\begin{enumerate}
\item[(i)] If $Q$ is a loop, let $\mathcal W = \{A^{\cdot\cdot}_{x,y}$, $B^{\cdot\cdot}_{x,y}$, $A^{\cdot\ldiv}_{x,y}$, $C^\cdot_x\}$.
\item[(ii)] If $Q$ is an inverse property loop, let $\mathcal W = \{A^{\cdot\cdot}_{x,y}$, $C^\cdot_x\}$.
\item[(iii)] If $Q$ is a group, let $\mathcal W = \{C^\cdot_x\}$.
\item[(iv)] If $Q$ is a commutative loop, let $\mathcal W = \{A^{\cdot\cdot}_{x,y}$, $A^{\cdot\ldiv}_{x,y}\}$.
\end{enumerate}
Then
\begin{displaymath}
    [\alpha,\beta] = \Cg{ (W_{\bar u}(a), W_{\bar v}(a) );\;W\in\mathcal W,\,1\,\alpha\, a,\,\bar u\,\beta\,\bar v}.
\end{displaymath}
\end{corollary}
\begin{proof}
Note that $R_{x,y}=A^{\cdot\cdot}_{y,x}$, $L_{x,y}=B^{\cdot\cdot}_{y,x}$, $T_x=(C^\cdot_x)^{-1}$ and $M_{x,y}=A^{\cdot\ldiv}_{y,x}U_x$, so $\mathcal W\cup\{U_x\}$ does the job, by Corollary \ref{Cr:CongrCom}. We can remove $U_x$ by Proposition \ref{Pr:Pruning}.
\end{proof}

It follows from Theorem \ref{Th:FundamentalFinite} that for loops with a finiteness condition we can remove $A^{\cdot\ldiv}_{x,y}$ from the sets $\W$ of Corollary \ref{Cr:CongrComAC}.

\begin{problem}
Find all minimal subsets of the $12$ associators and $4$ commutators that, together with $M_1$, generate total inner mapping groups in all loops.
\end{problem}

\section{The commutator of normal subloops}\label{Sc:SubloopComm}

The correspondence $\alpha\mapsto N_\alpha$, $N\mapsto \gamma_N$ between loop congruences and normal subloops allows us to restate Theorem \ref{Th:CongrComInf} and all its corollaries in terms of normal subloops, rather than in terms of congruences.

\begin{lemma}\label{Lm:generators}
Let $Q$ be a loop.
\begin{enumerate}
	\item[(i)] If $X\subseteq Q\times Q$ and $\alpha=\mathrm{Cg}(X)$, then $N_\alpha=\Ng{x/y;\;(x,y)\in X}$.
	\item[(ii)] If $X\subseteq Q$ and $N=\mathrm{Ng}(X)$, then $\gamma_N=\Cg{(x,1);\;x\in X}$.
\end{enumerate}
\end{lemma}
\begin{proof}
(i) Let $N=\Ng{x/y;\;(x,y)\in X}$. Since $(x,y)\in X\subseteq\alpha$, we immediately get $N\subseteq N_\alpha$. On the other hand, since $(x/y,1)\in\gamma_N$ for every $(x,y)\in X$, we also get $(x,y)\in\gamma_N$ for every $(x,y)\in X$, hence $\alpha\leq\gamma_N$.
Part (ii) is similar.
\end{proof}

Applying this observation to Theorem \ref{Th:CongrComInf}, we immediately get a generating set for the commutator of two normal subloops, as described in Theorem \ref{Th:FundamentalS}, stating that
\begin{displaymath}
    [A,B]_Q = \Ng{W_{\bar u}(a)/W_{\bar v}(a);\;W\in\mathcal W,\,a\in A,\,\bar u/\bar v\in B }
\end{displaymath}
for any normal subloops $A$, $B$ of any loop $Q$ in a variety $\V$, where $\mathcal W$ is a set of words that generates total inner mapping groups in $\V$. Of course, in loops with a finiteness condition we only need $\W$ that generates inner mapping groups. Using Corollary \ref{Cr:CongrComAC}, we obtain generating sets consisting of quotients of certain associators and commutators. Let us discuss the case of groups first.

Let $Q$ be a group and $A$, $B\unlhd Q$. Note that $C_y^\cdot(x) = y^{-1}xy$ and $c^\cdot(x,y) = x^{-1}y^{-1}xy = [x,y]$. Corollary \ref{Cr:CongrComAC} therefore yields
\begin{displaymath}
    [A,B]_Q=\Ng{[a,u]/[a,v];\;a\in A,\,u/v\in B}.
\end{displaymath}
From this it is not difficult to recover the standard group-theoretical result
\begin{displaymath}
    [A,B]_Q=\langle[a,b];\;a\in A,\,b\in B\rangle.
\end{displaymath}
Namely, let $N_1=\Ng{[a,u]/[a,v];\;a\in A,\,u/v\in B}$ and $N_2=\langle[a,b];\;a\in A,\,b\in B\rangle$. First notice that $N_2$ is a normal subgroup, since $\inn{Q}=\langle T_x;\;x\in Q\rangle\le\aut{Q}$ and thus $T_x([a,b]) = [T_x(a),T_x(b)]$ and $T_x(a)\in A$, $T_x(b)\in B$ for every $a\in A$, $b\in B$. Taking $u=b\in B$ and $v=1$, we obtain $[a,b]=[a,u]/[a,v]$ with $u/v\in B$, hence $N_2\subseteq N_1$. Conversely, calculating in $Q/N_2$, we get $[a,u]/[a,v] = a^{-1}u^{-1}auv^{-1}a^{-1}va = 1$, since we can commute $uv^{-1}\in B$ with $a\in A$ and cancel; this shows $N_1\subseteq N_2$.

The discussion of the group case raises two natural questions for general loops. First, is the subloop generated by the quotients always normal?
Second, can we dispose of the quotients in the generating set of $[A,B]_Q$?

Let us first answer the first question. In general, normality fails, as the following example and Corollary \ref{Cr:CongrCom} illustrate:
\begin{example}
Let $Q$ be the commutative loop
\begin{displaymath}
    \begin{array}{c|cccccccc}
    &0&1&2&3&4&5&6&7\\
    \hline
    0&0&1&2&3&4&5&6&7\\
    1&1&0&3&2&5&4&7&6\\
    2&2&3&1&0&6&7&4&5\\
    3&3&2&0&1&7&6&5&4\\
    4&4&5&6&7&3&0&1&2\\
    5&5&4&7&6&0&3&2&1\\
    6&6&7&4&5&1&2&3&0\\
    7&7&6&5&4&2&1&0&3
    \end{array}
\end{displaymath}
Then $A=\{0,1,2,3\}$ is a normal subloop of $Q$. But $\langle L_{u_1,u_2}(a)/L_{v_1,v_2}(a);\;a\in A,\,\bar u/\bar v\in A\rangle = \{0,1\}$ is not a normal subloop of $Q$.
\end{example}

On the other hand, in many loops, the answer is positive. The gist of the proof for groups was the fact that inner mappings of groups are automorphisms. Recall that a loop $Q$ is said to be automorphic if $\inn{Q}\le\aut{Q}$.

\begin{proposition}\label{Pr:automorphic}
Let $\V$ be a variety of automorphic loops and $\mathcal W$ a set of words that generates total inner mapping groups in $\V$. Then
\begin{displaymath}
    [A,B]_Q = \langle W_{\bar u}(a)/W_{\bar v}(a);\;W\in\mathcal W,\,a\in A,\,\bar u/\bar v\in B\rangle
\end{displaymath}
for any normal subloops $A$, $B$ of any $Q\in\V$.
\end{proposition}
\begin{proof}
Denote the right hand side subloop by $N$. In view of Theorem \ref{Th:FundamentalS}, we only need to check that $N$ is normal in $Q$. Since inner mappings are automorphisms, we only need to check that the generators of $N$ are preserved by inner mappings. Let $F$ be any of the words $L_{x,y},R_{x,y},T_x$, and let $W\in\W$. Then the composition $FW$ is also an inner word. Hence, by Lemma \ref{Lm:CongrComm-generators}, $(F_{\bar x}W_{\bar u}(a),F_{\bar x}W_{\bar v}(a))\in[\gamma_A,\gamma_B]$ for every tuple $\bar x$ over $Q$, $a\in A$ and tuples $\bar u$, $\bar v$ over $Q$ such that $\bar u/\bar v\in B$. Thus $F_{\bar x}(W_{\bar u}(a)/W_{\bar v}(a))=F_{\bar x}W_{\bar u}(a)/F_{\bar x}W_{\bar v}(a)\in[A,B]_Q$.
\end{proof}

\begin{problem}
Characterize loops $Q$ such that (with the notation of Theorem \ref{Th:NormComInf}) the subloop $\langle W_{\bar u}(a)/W_{\bar v}(a);\;W\in\mathcal W,\,a\in A,\,\bar u/\bar v\in B\rangle$ is normal for all subloops $A$, $B\unlhd Q$.
\end{problem}

The second question, whether quotients can be reduced, is also tricky. Example \ref{Ex:minus} shows that it is not possible to get rid of quotients in the standard generating set, or in the generating set resulting from elementwise associators and commutators. (It might be possible to get rid of all quotients in different generating sets.) On the other hand, Proposition \ref{Pr:Pruning} says that some quotients can be removed: words without parameters for good, and the words $U_x$ and $T_x$ can be replaced by certain associators and commutators.

\begin{corollary}\label{Cr:FundamentalS}
Let $Q$ be a loop and $A,B$ normal subloops of $Q$.
\begin{itemize}
	\item[(i)] Then
\begin{align*}
    [A,B]_Q = \Ng{ [a,b],\,&[b,a,x],\,W_{u_1,u_2}(a)/W_{v_1,v_2}(a);\\&W\in\{L,R,M\},\,a\in A,\,b\in B,\,x\in Q,\,\bar u/\bar v\in B }.
\end{align*}
	\item[(ii)] If $Q$ is an inverse property loop, then
$$    [A,B]_Q = \Ng{ [a,b],\,L_{u_1,u_2}(a)/L_{v_1,v_2}(a);\;a\in A,\,b\in B,\,\bar u/\bar v\in B }.$$
	\item[(iii)] If $Q$ is a group, then
$$    [A,B]_Q = \langle [a,b];\; a\in A,\,b\in B \rangle.$$
	\item[(iv)] If $Q$ is a commutative loop, then
$$    [A,B]_Q = \Ng{ W_{u_1,u_2}(a)/W_{v_1,v_2}(a);\;W\in\{L,M\},\,a\in A,\,\bar u/\bar v\in B }.$$
\end{itemize}
\end{corollary}
\begin{proof}
Corollary \ref{Cr:CongrCom} can be translated via Lemma \ref{Lm:generators} in the same way that we have translated Theorem \ref{Th:CongrComInf} into Theorem \ref{Th:NormComInf}. We will use the translation of Corollary \ref{Cr:CongrCom} without reference. In all cases, let $N$ denote the subloop on the right hand side.

(i) We check that $N$ satisfies the assumptions of conditions (ii), (iii) of Proposition \ref{Pr:Pruning}. We will calculate in $Q/N$. For every $a\in A$, $b\in B$, $x\in Q$, we have
$R_{x,b}(a)=R_{x,1}(a)=a$, so $[a,x,b]=1$, and $R_{x,b}(a) = R_{x,1}(a) = a$, so $[a,b,x]=1$, and also
$L_{x,b}(a)=L_{x,1}(a)=a$, so $[x,b,a]=1$.

(ii) Following the proof of case (i), we only need to show that $[b,a,x]\in N$ for every $a\in A$, $b\in B$, $x\in Q$. The following statements, universally quantified with $a\in A$, $b\in B$, $x\in Q$, are equivalent: $b\cdot ax = ba\cdot x$, $ax = b^{-1}(ba\cdot x)$, (use substitution $x\mapsto (ba)^{-1}x$) $a\cdot (ba)^{-1}x = b^{-1}x$, $(ba)^{-1}x = a^{-1}\cdot b^{-1}x$, (use the AAIP) $[a^{-1},b^{-1},x]=1$.

(iii) This follows from case (ii) once we realize that $L_{x,y}=1$. Proposition \ref{Pr:automorphic} applies.

(iv) This is merely a restatement of Corollary \ref{Cr:CongrCom}(iv).
\end{proof}

In loops with a finiteness condition we can omit the mappings $M_{x,y}$ from the sets $\W$ of Corollary \ref{Cr:FundamentalS}.

In the proof, we rely on the ad hoc arguments of Proposition \ref{Pr:Pruning}. We therefore ask:

\begin{problem}
Describe systematically when quotients of inner mappings can be reduced, analogously to Corollary \ref{Cr:FundamentalS}.
\end{problem}

\section{The derived subloop and the associator subloop}\label{Sc:ASDS}

Recall that the derived subloop $Q'$ of $Q$ is the smallest normal subloop of $Q$ such that $Q/Q'$ is an abelian group,
and the associator subloop $A(Q)$ of a loop $Q$ is the smallest normal subloop of $Q$ such that $Q/A(Q)$ is a group.
It is clear that $$Q'=\Ng{[x,y,z],\,[x,y];\;x,y,z\in Q} \quad\text{ and }\quad A(Q)=\Ng{[x,y,z];\;x,y,z\in Q}.$$
Alternatively, we can use the associators and commutators defined in Section \ref{Sc:AC}, along the guidelines given by Lemma \ref{Lm:Factors}.

As in groups, it was shown by Bruck \cite[p. 13]{B} that, in fact,
$Q'=\langle [x,y,z],\,[x,y];\;x,y,z\in Q\rangle$. However, the case of
$A(Q)$ is more complicated: $\langle [x,y,z];\;x,y,z\in Q\rangle$
needs not be normal in $Q$; one has to consider its normal closure. We
are going to see that the normal closure is not needed upon replacing
the traditional associators/commutators with our associators/commutators.


\begin{lemma}\label{Lm:Q'}
Let $\V$ be a variety of loops, $\W$ a set of words generating inner mapping groups in $\V$, and $N\unlhd Q\in\V$. The following two conditions are equivalent:
\begin{enumerate}
	\item[(i)] $Q/N$ is an abelian group.
	\item[(ii)] $W_{\bar x}(z)/z\in N$ for every $W\in\W$, $\bar x$ a tuple over $Q$, and $z\in Q$.
\end{enumerate}
\end{lemma}
\begin{proof}
Condition (ii) says that, in $Q/N$, $W_{\bar x}=1$ for every $W\in\W$ and $x_1,\dots,x_n\in Q/N$. Since the mappings $W_{\bar x}$ generate $\inn{Q/N}$, this is equivalent to the fact that $\inn{Q/N}=1$. This is equivalent to $Q/N$ being an abelian group, i.e., (i).
\end{proof}

\begin{theorem}\label{Th:Q'}
Let $\V$ be a variety of loops, $\W$ a set of words generating inner mapping groups in $\V$, and $Q\in\V$. Then
\begin{displaymath}
    Q'=\langle W_{\bar x}(z)/z;\;W\in\W,\,\bar x\text{ a tuple over }Q,\,z\in Q\rangle.
\end{displaymath}
\end{theorem}

\begin{proof}
Let $H$ denote the subloop on the right hand side. In view of Lemma \ref{Lm:Q'}, it suffices to show that $H\unlhd Q$. Using Proposition \ref{Pr:Inn}(ii), we only need to check that $W_{\bar x}(H)=H$ for every $W\in\W$ and every tuple $\bar x$ over $Q$. For $a\in H$, we have $W_{\bar x}(a)/a\in H$ by definition, so $W_{\bar x}(a) = (W_{\bar x}(a)/a)\cdot a\in H$.
\end{proof}

Any choice of associators and commutators such that the set of the corresponding inner words generates inner mapping groups will provide a generating set for $Q'$. Here is such a choice, corresponding to the standard generating set of $\inn Q$. (The facts of Corollary \ref{Cr:Q'} were observed by Covalschi and Sandu in \cite{CS}.)

\begin{corollary}\label{Cr:Q'}\
\begin{enumerate}
	\item[(i)] Let $Q$ be a loop. Then
\begin{align*}
Q'&=\langle L_{x,y}(z)/z,\,R_{x,y}(z)/z,\,T_x(z)/z;\;x,y,z\in Q\rangle\\
&=\langle a^{\cdot\cdot}(x,y,z),\,b^{\cdot\cdot}(x,y,z),\,c^{\cdot}(x,y);\;x,y,z\in Q\rangle.
\end{align*}
	\item[(ii)] Let $Q$ be an inverse property loop. Then
$$Q'=\langle L_{x,y}(z)/z,\,T_x(z)/z;\;x,y,z\in Q\rangle=\langle a^{\cdot\cdot}(x,y,z),\,c^{\cdot}(x,y);\;x,y,z\in Q\rangle.$$
	\item[(iii)] Let $Q$ be a group. Then
$$Q'=\langle T_x(y)/y;\;x,y\in Q\rangle=\langle [x,y];\;x,y\in Q\rangle.$$
\end{enumerate}
\end{corollary}
\begin{proof}
Note that if $N\unlhd Q$ then $a/b\in N$ iff $a\ldiv b\in N$. It is therefore irrelevant on which side of the inner mappings $W(z)$ we divide by $z$.
\end{proof}

Notice that we have just recovered the classical result of group theory that $Q'$ is the subgroup generated by all commutators.

The case of the associator subloop is more difficult. A similar trick as above allows us to show that the subloop is preserved by the inner mappings $L_{x,y}$ and $R_{x,y}$, but it cannot be used for $T_x$, because $A(Q)$ does not contain commutators. The idea behind the proof of Theorem \ref{Th:A(Q)} comes from Leong \cite{L}, who proved a similar result with a different choice of associators, described in Section \ref{Sc:summary}. We will imitate his proof with our associators.

\begin{theorem}\label{Th:A(Q)}
Let $Q$ be a loop. Then $A(Q) = \langle a^{\cdot\cdot}(x,y,z)$, $b^{\cdot\cdot}(x,y,z);\;x,y,z\in Q\rangle$.
\end{theorem}
\begin{proof}
Write $a$ instead of $a^{\cdot\cdot}$, $b$ instead of $b^{\cdot\cdot}$, and let $H$ be the subloop on the right hand side. By Lemma \ref{Lm:Factors}, it suffices to show that $H\unlhd Q$.

For $h\in H$ we have $(hx)y = a(h,x,y)h\cdot xy\in H\cdot xy$ and $x\cdot yh = xy\cdot hb(h,y,x)\in xy\cdot H$, so
$$Hx\cdot y\subseteq H\cdot xy\quad\text{ and }\quad x\cdot yH\subseteq xy\cdot H.$$
We will use these inclusions freely.

We claim that
\begin{equation}\label{Eq:Cosets}
    (H/x)/y = H/(yx).
\end{equation}
Let $h\in H$. For one inclusion, we need to show that $((h/x)/y)\cdot yx\in H$. Now, $((h/x)/y)\cdot yx = ((h/x)/y)y\cdot xb((h/x)/y,y,x) \in ((h/x)/y)y\cdot xH = (h/x)\cdot xH\subseteq (h/x\cdot x)H = hH = H$. For the other inclusion, we need to show that $(h/(yx))y\cdot x\in H$. Now, $(h/(yx))y\cdot x = a(h/(yx),y,x)(h/(yx))\cdot yx \in (H\cdot h/(yx))\cdot yx \subseteq H((h/(yx))\cdot yx) = Hh = H$. We will use \eqref{Eq:Cosets} freely.

Let $X = \{H/x;\;x\in Q\}$. For $x\in Q$ define $\alpha_x\in \sym{X}$ by $$\alpha_x(H/y) = (H/y)/x = H/(xy).$$
If $H/y=H/z$ then $(H/y)/x = (H/z)/x$, so $\alpha_x$ is a well-defined mapping into $X$. If $(H/y)/x = (H/z)/x$ then $H/y = H/z$, so $\alpha_x$ is one-to-one. Finally, for any $y\in Q$ we have $\alpha_x(H/(x\ldiv y)) = H/(x\cdot x\ldiv y) = H/y$, so $\alpha_x$ is onto $X$.

Note that $y\cdot zx\in yz\cdot xH\subseteq (yz\cdot x)H$, so there is $h\in H$ such that $y\cdot zx = (yz\cdot x)h$. Consequently, $H/(y\cdot zx) = H/((yz\cdot x)h) = (H/h)/(yz\cdot x) = H/(yz\cdot x)$. Define $$\alpha:Q\to \sym{X},\quad x\mapsto\alpha_x.$$
Then $\alpha_{yz}(H/x) = H/(yz\cdot x) = H/(y\cdot zx) = (H/(zx))/y = ((H/x)/z)/y = \alpha_y\alpha_z(H/x)$, so $\alpha$ is a homomorphism into a group.
Let $K=\ker(\alpha)$. Then $Q/K \cong \img(\alpha)$ is a group, and so $H\subseteq K$ by Lemma \ref{Lm:Factors}. Given $x\in K$, we have $(H/y)/x = H/y$ for every $y\in Q$, in particular, with $y=1$ we get $H/x = H$, so $x\in H$, proving $K\subseteq H$. This means that $H=K\unlhd Q$.
\end{proof}

\begin{corollary}\label{Cr:A(Q)}
Suppose that $Q$ is a finite loop, or an inverse property loop, or a commutative loop. Then
$$A(Q) = \langle a^{\cdot\cdot}(x,y,z);\;x,y,z\in Q\rangle = \langle b^{\cdot\cdot}(x,y,z);\;x,y,z\in Q\rangle.$$
\end{corollary}
\begin{proof}
If $Q$ is an inverse property loop, observe that $a^{\cdot\cdot}(x,y,z)^{-1} = b^{\cdot\cdot}(x^{-1},y^{-1},z^{-1})$. If $Q$ is commutative, we have $a^{\cdot\cdot}(x,y,z) = b^{\cdot\cdot}(x,y,z)$. In either case, we are done by Theorem \ref{Th:A(Q)}.

Now suppose that $Q$ is a finite loop, and let us focus on the associator $b = b^{\cdot\cdot}$. Let $H=\langle b(x,y,z);\;x,y,z\in Q\rangle$. As in the proof of Theorem \ref{Th:A(Q)}, we have $x\cdot yH\subseteq xy\cdot H$. Since the two sets have the same cardinality, we have $x\cdot yH = xy\cdot H$ by finiteness. Using this inclusion with $h\in H$, we have $((h/x)/y)\cdot yx \in ((h/x)/y)y\cdot xH = (h/x)\cdot xH\subseteq (h/x\cdot x)H = hH = H$, so $(H/x)/y\subseteq H/(yx)$, and thus $(H/x)/y = H/(yx)$ by finiteness. Then the last two paragraphs of the proof of Theorem \ref{Th:A(Q)} go through word for word.
\end{proof}

Our choice of associators and commutators for Theorem \ref{Th:A(Q)} is certainly not the only choice; in the end, Leong used a different associator $b$.

\begin{problem}
Determine all minimal subsets $\mathcal A$ of the eight types of associators of Lemma \ref{Lm:Factors}(i) such that for every loop $Q$ we have $A(Q) = \langle a(x,y,z);\;a\in \mathcal A$, $x$, $y$, $z\in Q\rangle$.
\end{problem}

Just like in group theory, it is reasonable to consider the smallest normal subloop $N$ such that $Q/N$ is a commutative loop. Clearly, $N=\Ng{[x,y];\;x,y\in Q}$. Can we possibly avoid the normal closure? Let
\begin{displaymath}
    \mathrm{Comm}(Q) = \langle c^{\cdot}(x,y),\,d^{\cdot}(x,y),\,c^{\ldiv}(x,y),\,d^{\rdiv}(x,y);\;x,y\in Q\rangle.
\end{displaymath}
It is easy to check that $c^{\ldiv}(x,y)=c^{\cdot}(x,x\ldiv y)$ and $d^{\rdiv}(x,y)=d^{\cdot}(x,y\rdiv x)$,
hence, in fact, $\mathrm{Comm}(Q) = \langle c^{\cdot}(x,y),\,d^{\cdot}(x,y);\;x,y\in Q\rangle$. Unfortunately, $\mathrm{Comm}(Q)$ is not necessarily normal in $Q$, as witnessed by Example \ref{Ex:Comm}.

\begin{problem}
Is there a loop $Q$ such that $\langle a^{\cdot\cdot}(x,y,z);\;x,y,z\in Q\rangle \ne \langle b^{\cdot\cdot}(x,y,z);\;x,y,z\in Q\rangle$? Is there a loop $Q$ such that $\langle c^\cdot(x,y);\;x,y\in Q\rangle\ne \langle d^\cdot(x,y);\;x,y\in Q\rangle$?
\end{problem}

\section{Examples and counterexamples}\label{Sc:Examples}

\def\ordiv{\oslash}
\def\oldiv{\,\rotatebox[origin=c]{90}{$\ordiv$}\,}

All examples in this section will be based on a general construction, inspired by the example in \cite[Chapter 5, Exercise 10]{FM}.

\begin{construction}\label{FMexample}
Let $(G,+)$ be an abelian group and let $(G,\oplus)$ be a quasigroup. We define $Q=G[\oplus]$ to be the loop on $G\times \Z_2$ with multiplication
\begin{displaymath}
    (x,a)(y,b)=\left\{
	\begin{array}{ll}
		(x+y,a+b)  & \text{if $a=0$ or $b=0$,}\\
		(x\oplus y,0) & \text{otherwise.}
	\end{array}\right.
\end{displaymath}
\end{construction}

\emph{Properties.} The set $H=G\times\{0\}$ forms a normal subloop isomorphic to $(G,+)$, since it is the kernel of the projection $Q\to\Z_2$, $(x,a)\mapsto a$. Hence $Q$ possesses a chain of normal subloops $0\leq H\leq Q$ such that each factor is an abelian group. Consequently, $Q'\leq H$ and $Q$ is solvable (in the sense of Bruck).

\emph{Notational remarks.}
For $x\in Q$, we will implicitly assume $x=(x_0,x_1)$. For brevity, for $a\in H$ and $k\in\mathbb N$, let $ka=(ka_0,0)$. The division operations with respect to $\oplus$ will be denoted by $\oldiv$ and $\ordiv$. The identity element $(0,0)$ of $Q$ will be denoted by $0$.

\emph{Inner mappings.} Since $Q'\subseteq H$, to determine congruence solvability and nilpotency we will need to calculate commutators $[A,B]_Q$ for $A$, $B\unlhd Q$ such that $A\subseteq H$. Hence, we need to determine the values of inner mappings on the elements of $H$.
For every $a,b\in H$, $x,y\in Q\smallsetminus H$ and every $z\in Q$, we get
\begin{align*}
	T_z(a)&=a\quad\text{and}\quad U_z(a)=-a,\\
	L_{b,z}(a)&=L_{z,b}(a)=R_{b,z}(a)=R_{z,b}(a)=a\quad\text{and}\quad M_{z,b}(a)=M_{x,z}(a)=-a,\\
	L_{x,y}(a)&=((x_0\oplus(y_0+a_0))-(x_0\oplus y_0),0),\\
	R_{x,y}(a)&=(((y_0+a_0)\oplus x_0)-(y_0\oplus x_0),0),\\
	M_{b,x}(a)&=((x_0\oldiv b_0)-((x_0-a_0)\oldiv b_0),0).
\end{align*}
It follows from Theorem \ref{Th:FundamentalS} and Proposition \ref{Pr:TotInnGens} that
\begin{displaymath}
    [A,B]_Q = \Ng{ W_{u_1,u_2}(a)/W_{v_1,v_2}(a);\;W\in\{L,R,M\},\,a\in A,\,\bar u/\bar v\in B }
\end{displaymath}
for every $A$, $B\unlhd Q$ such that $A\subseteq H$.

Also note that $c^\cdot(z,a)=1$ and $a^{\cdot\cdot}(a,b,z)=a^{\cdot\cdot}(a,z,b)=b^{\cdot\cdot}(a,b,z)=b^{\cdot\cdot}(a,z,b)=1$ whenever $a,b\in H$ and $z\in Q$.
\bigskip

The next three examples show how to use our theory to efficiently calculate commutators and derived subloops, and also illustrate the variety of options that Construction \ref{FMexample} offers. Imitating the notation for congruences, set
\begin{displaymath}
    Q^{(0)}=Q_{(0)}=Q, \qquad Q_{(i+1)}=[Q_{(i)},Q]_Q, \qquad Q^{(i+1)}=[Q^{(i)},Q^{(i)}]_Q.
\end{displaymath}
Note that $Q^{(1)}=Q_{(1)}=Q'$. A loop $Q$ is called \emph{centrally nilpotent} if $Q_{(n)}=1$ for some $n$, and it is called
\emph{congruence solvable} if $Q^{(n)}=1$ for some $n$.


\begin{example}\label{Ex:minus}
Let $(G,+)$ be an abelian group. Consider the loop $Q=G[-]$, i.e., $x\oplus y=x-y$, the subtraction in $G$. In general we obtain a non-commutative non-associative loop. It is easy to check that $x\oldiv y=x-y$ and $x\ordiv y=x+y$. For $n\in\mathbb N$, let $H_n=2^nG\times\{0\}$, where $mG=\{mg;\;g\in G\}$, and notice that this is a subloop of $Q$.

Let $a,b\in H$ and $x,y\in Q\smallsetminus H$.
Using the general expressions above, we see that $L_{x,y}(a)=((x_0-(y_0+a_0))-(x_0-y_0),0)=(-a_0,0)=-a$, and thus also $L_{x,y}^{-1}(a)=-a$. Similar computations yield $R_{x,y}(a)=a$ and $M_{b,x}(a)=a$, and the remaining inner mappings are also identical or inverse mappings on $H$.
Consequently, every subloop of $H$ is normal in $Q$ (in particular, $H_n\unlhd Q$), and
\begin{displaymath}
    [A,B]_Q = \Ng{ L_{u_1,u_2}(a)/L_{v_1,v_2}(a);\;\,a\in A,\,\bar u/\bar v\in B }
\end{displaymath}
for every $A,B\unlhd Q$ such that $A\subseteq H$.

Let us calculate the derived subloop of $Q=G[-]$. Note that $a^{\cdot\cdot}(x,y,z)$, $b^{\cdot\cdot}(x,y,z)$, and $c^\cdot(x,y)$, evaluated in $Q$, are expressions with an even number of occurences of $x_0,y_0,z_0$, each with a positive or negative sign. For example,
\begin{displaymath}
    c^\cdot((x_0,1),(y_0,1))=((y_0\oldiv(x_0\oplus y_0))-x_0,0)=((y_0-(x_0- y_0))-x_0,0)=(2y_0-2x_0,0).
\end{displaymath}
Consequently, the result is always in $H_1=2G\times\{0\}$. On the other hand, every element $2a\in H_1$ can be expressed, for example, by $2a=(2a_0,0)=c^\cdot((a,1),(0,1))$. We see that $Q'=H_1$.

Now, let us have a look at the commutator. If both $A,B\subseteq H$, then $[A,B]_Q\subseteq[H,H]_Q$. But $[H,H]_Q=0$: if $u_i/v_i\in H$, then $u_i\in H$ iff $v_i\in H$, hence
either $L_{u_1,u_2}(a)/L_{v_1,v_2}(a)=a/a=0$, or $L_{u_1,u_2}(a)/L_{v_1,v_2}(a)=(-a)/(-a)=0$. In particular, $Q^{(2)}=0$ and $Q$ is congruence solvable.

On the other hand, $[A,Q]_Q$ may not vanish, since $L_{(0,1),(0,1)}(a)/L_{(0,0),(0,0)}(a)=(-a)/a=-2a$, hence $[A,Q]_Q=2A$.
Consequently, we have $Q_{(n)}=H_n$ for every $n\geq1$. If $|G|$ is odd, we have $Q_{(n)}=H$ for every $n$, and $Q$ is not centrally nilpotent. If $G=\Z$, for instance, we obtain a strictly decreasing chain with trivial intersection, hence $Q$ is not centrally nilpotent (this situation is sometimes refered to as transfinite nilpotency). If $|G|$ is a power of two, then $2^nG=0$ for some $n$, and thus $Q$ is centrally nilpotent.
\end{example}

Example \ref{Ex:minus} also shows an obstacle to removing quotients from the generating set of the commutator, as described in Corollary \ref{Cr:FundamentalS}. Let $A,B\unlhd G[\oplus]$ be such that $A\subseteq H$. On one hand, all associators and commutators with at least one parameter from $A$ and one parameter from $B$ vanish. On the other hand, associators with one parameter in $A$ (or $B$) and other parameters arbitrary may not belong to $[A,B]_Q$. For instance, in $G[-]$ with $A=B=H$, we have $[H,H]_Q=0$, but $b^{\cdot\cdot}(a,x,y)=a\ldiv L_{y,x}(a)=-2a$, for every $a\in H$ and $x,y\in Q\smallsetminus H$.

The next two examples show a particular choice of $G$ and $\oplus$ such that the subloop $H$, which itself is an abelian group, is not abelian in $G[\oplus]$.
In Example \ref{Ex:Z4}, $G[\oplus]$ is not congruence solvable. This is a rather typical situation, resulting from most combinations of $G$ and $\oplus$.
In Example \ref{Ex:Z22}, $G[\oplus]$ is congruence solvable.

\begin{example}\label{Ex:Z4}
Consider the loop $Q=\Z_4[\oplus]$, where the operation $\oplus$ is given by the following multiplication table:
$$\begin{array}{c|cccc}
 & 0&1&2&3 \\ \hline 0&0&1&2&3\\1&1&3&0&2\\2&2&0&3&1\\3&3&2&1&0
\end{array}$$
Notice that $(\{0,1,2,3\},\oplus)\cong\Z_4$. We will show that $Q'=H$ and that $Q^{(2)}=Q_{(2)}=H$, hence $Q$ is not congruence solvable.

Observe that $L_{(0,1),(0,1)}((1,0))=(1,0)$ and $L_{(1,1),(0,1)}((1,0))=(2,0)$. It follows that $b^{\cdot\cdot}((1,0),(0,1),(1,1))=(1,0)\ldiv L_{(1,1),(0,1)}((1,0))=(1,0)\in Q'$, and thus $H=\langle(1,0)\rangle\subseteq Q'\subseteq H$, hence $Q'=H$.
Furthermore, $L_{(1,1),(0,1)}((1,0))/L_{(0,1),(0,1)}((1,0))=(1,0)\in[H,H]_Q$, hence $H=\langle(1,0)\rangle\subseteq [H,H]_Q\subseteq[H,Q]_Q\subseteq H$, so $[H,H]_Q=[H,Q]_Q=H$.

According to \texttt{GAP}, the total multiplication group $\totmlt Q$ is solvable. Hence, Vesanen's theorem \cite{V} does not strengthen to congruence solvability.
\end{example}

\begin{example}\label{Ex:Z22}
Consider the loop $Q=\Z_4[\oplus]$, where the operation $\oplus$ is given by the following multiplication table:
$$\begin{array}{c|cccc}
 & 0&1&2&3 \\ \hline 0&0&1&2&3\\1&1&0&3&2\\2&2&3&0&1\\3&3&2&1&0
\end{array}$$
Notice that $(\{0,1,2,3\},\oplus)\cong\Z_2\times\Z_2$ and that
\begin{displaymath}
    r\oplus s=\left\{
	\begin{array}{ll}
		r+s+2  & \text{if } r,s\in\{1,3\},\\
		r+s & \text{otherwise.}
	\end{array}\right.
\end{displaymath}
Consequently, we can write $r\oplus s=r+s+\varepsilon$ where $\varepsilon\in\{0,2\}$, where $\varepsilon=2$ iff $r$, $s\in\{1,3\}$, and similarly for the division operations $\oldiv$, $\ordiv$. Let $K=\{0,2\}\times\{0\}$. We will show that $Q'=[H,H]_Q=K$, and that $Q^{(2)}=Q_{(2)}=0$. Hence, $Q$ is centrally nilpotent, although $H$ is not abelian in $Q$.

$Q$ is finite and commutative, so we can focus on $L_{x,y}$. For $a\in H$ and $x,y\in Q\smallsetminus H$, we have
$$L_{x,y}(a)=((x_0\oplus(y_0+a))-(x_0\oplus y_0),0)=(x+(y+a)+\varepsilon_1-(x+y+\varepsilon_2),0)=(\varepsilon_1-\varepsilon_2,0)\in K.$$
We immediately see that $K$ is normal in $Q$ and that $[H,H]_Q\subseteq K$. Since
$$L_{(1,1),(0,1)}(1,0)/L_{(0,1),(0,1)}(1,0)=(3,0)/(1,0)=(2,0)\in[H,H]_Q,$$
we obtain $[H,H]_Q=K$. Furthermore, if $a\in K$, then $\varepsilon_1=\varepsilon_2$, and thus $[K,K]_Q\subseteq[K,Q]_Q=0$.
To show that $Q'=K$, calculate the associator $a^{\cdot\cdot}(x,y,z)=((xy\cdot z)/(yz))/x=((((x_0+y_0+\varepsilon_1)+z_0+\varepsilon_2)-(y_0+z_0+\varepsilon_3)+\varepsilon_4)-x+\varepsilon_5,0)=(\varepsilon_1+\varepsilon_2-\varepsilon_3+\varepsilon_4+\varepsilon_5,0)\in K$, so $Q'\subseteq K$.
Since $Q$ is not associative, we get $Q'=K$.
\end{example}

The distinction between ``abelianess'' and ``abelianess in $Q$'' persists even in varieties of loops that are very close to groups. For instance, let $Q$ be the left Bol loop of order $8$ from Example \ref{Ex:TotInnGens} (catalog number \texttt{LeftBolLoop(8,1)} in the \texttt{LOOPS} package for \texttt{GAP}), and let $N=\{1,2,3,4\}$. Then $N\unlhd Q$, $N$ is an abelian group, but $[N,N]_Q=Z(Q)=\{1,2\}$, so $N$ is not abelian in $Q$. Nevertheless, $Q$ is congruence solvable. A similar situation occurs in the Moufang loop of order $16$ with catalog number \texttt{MoufangLoop(16,4)} for one of its normal subloops isomorphic to $\Z_2\times\Z_4$.

Finally, we present an example of a loop where $\mathrm{Comm}(Q)$ is not normal.

\begin{example}\label{Ex:Comm}
Consider the loop $Q=\Z_4[\oplus]$, where the operation $\oplus$ is given by
\begin{displaymath}
    \begin{array}{c|cccc}
        \oplus&0&1&2&3\\
        \hline
        0&0&2&3&1\\
        1&1&0&2&3\\
        2&3&1&0&2\\
        3&2&3&1&0
    \end{array}
\end{displaymath}
One can check that $\mathrm{Comm}(Q) = \{(0,0),(2,0)\}$, which is not a normal subloop of $Q$.
\end{example}

\section{Center and nilpotency, abelianess and solvability}\label{Sc:univalg}

The purpose of this section is to supply details for the exposition in Section \ref{Sc:summary}.

Recall the definitions of the commutator $[\alpha,\beta]$, the center $\zeta(\A)$ and abelianess in $\A$.
First, we show how $[\alpha,1_\A]$ relates to the center, and how $[\alpha,\alpha]$ relates to abelianess.

Let $\alpha$ be a congruence of an algebra $\A$.
By definition, $[\alpha,1_\A]$ is the smallest congruence $\delta$ such that $C(\alpha,1_\A;\delta)$ in $\A$, or equivalently,
$C(\alpha/\delta,1_{\A/\delta};0_{\A/\delta})$ in $\A/\delta$. The definition of the center says $C(\xi,1_\B;0_\B)$ in $\B$ iff $\xi\leq\zeta(\B)$. Applied to $\B=\A/\delta$, we see that $[\alpha,1_\A]$ is the smallest $\delta$ such that $\alpha/\delta\leq\zeta(\A/\delta)$.

By definition, $[\alpha,\alpha]$ is the smallest congruence $\delta$ such that $C(\alpha,\alpha;\delta)$ in $\A$, or equivalently,
$C(\alpha/\delta,\alpha/\delta;0_{\A/\delta})$ in $\A/\delta$. The latter says that the congruence $\alpha/\delta$ is abelian in $\A/\delta$.

\subsection{In loops}

Recall that the center $Z(Q)$ of a loop $Q$ is defined in loop theory as
\begin{displaymath}
    Z(Q) = \{a\in Q;\;ax{=}xa,\,a(xy){=}(ax)y,\,x(ay){=}(xa)y,\,x(ya){=}(xy)a\text{ for every }x,y\in Q\}.
\end{displaymath}
Theorem \ref{Th:Center} shows that $Z(Q)$ and $\zeta(Q)$ define the same concept. For groups, the proof can be found in \cite[Section II.13]{BS}, for instance, and it easily extends to loops. For the sake of completeness (and because we are not aware of a proof in the literature), we present a complete proof here. It is instructive to read the proof to become accustomed to the universal algebraic approach to loop theory.

\begin{theorem}\label{Th:Center}
If $Q$ is a loop, then $Z(Q) = N_{\zeta(Q)}$.
\end{theorem}
\begin{proof}
We will prove two inclusions: $\gamma_{Z(Q)}\subseteq\zeta(Q)$ and $N_{\zeta(Q)}\subseteq Z(Q)$.

Let $a\,\gamma_{Z(Q)}\,b$. We want to show $a\,\zeta(Q)\,b$. Since $\zeta(Q)$ is the largest congruence such that $C(\zeta(Q),1_Q;0_Q)$, it is sufficient to show that (the congruence generated by) the pair $(a,b)$ centralizes $1_Q$ over $0_Q$. Let $t$ be a term and $u_1,\dots,u_n$, $v_1,\dots,v_n$ two tuples over $Q$. Assuming $t(a,u_1,\dots,u_n)=t(a,v_1,\dots,v_n)$, we get
\begin{multline*}
    t(b,u_1,\dots,u_n) =t(b/a\cdot a,u_1,\dots,u_n) = b/a\cdot t(a,u_1,\dots,u_n) \\
    = b/a\cdot t(a,v_1,\dots,v_n) =t(b/a\cdot a,v_1,\dots,v_n) =t(b,v_1,\dots,v_n),
\end{multline*}
where the second and the fourth equalities follow form the fact that $b/a\in Z(Q)$, i.e., $b/a$ commutes and associates with everything.

Conversely, let $a\in N_{\zeta(Q)}$, i.e., $a\,\zeta(Q)\,1$. We want to show that $a\in Z(Q)$. It actually suffices to show that $ab=ba$, $a(bc) = (ab)c$ and $b(ca) = (bc)a$ for every $b$, $c\in Q$, since then $b(ac) = b(ca) = (bc)a = a(bc) = (ab)c = (ba)c$, too. We will use $C(\zeta(Q),1_Q;0_Q)$ freely, noting that the equivalence modulo $0_Q$ is merely the equality.

For $ab=ba$, consider the term $t(x,y,z) = x(yz)$. Then $t(1,1,b)=b=t(b,1,1)$, and upon replacing the middle argument $1$ with $a$, we conclude that $ab = t(1,a,b) = t(b,a,1) = ba$. For $a(bc) = (ab)c$, consider the auxiliary term $m(x,y,z) = x(y\ldiv z)$. (This is in fact a Mal'tsev term for loops.) Then $m(1,a,a)m(b,1,c) = bc = m(1,1,b)m(c,c,c)$, and replacing the first argument $1$ with $a$ yields $a(bc) = m(a,a,a)m(b,1,c) = m(a,1,b)m(c,c,c) = (ab)c$. For $b(ca) = (bc)a$, we proceed dually and consider the auxiliary term $m'(x,y,z) = (x/y)z$. Then $m'(b,b,b)m'(c,1,1) = bc = m'(b,1,c)m'(a,a,1)$, and replacing the last argument $1$ with $a$ yields $b(ca) = m'(b,b,b)m'(c,1,a) = m'(b,1,c)m'(a,a,a) = (bc)a$.
\end{proof}

\begin{corollary}
A loop is abelian if and only if it is a commutative group.
\end{corollary}

\begin{proof}
$\zeta(Q)=1_Q$ iff $Z(Q)=Q$ iff $Q$ is commutative and associative.
\end{proof}

Using the observations at the beginning of the section, this finishes the proof that universal algebraic nilpotency is the same notion as central nilpotency.

We want to point out that abelian groups and nilpotent loops are important classes of algebras in the abstract structure theory of universal algebra. Let $A$ be an algebra with a Mal'tsev term $m$. Choose an arbitrary element $e\in A$ and define $a+b=m(a,e,b)$.
The fundamental theorem of abelian algebras \cite[Section 5]{FM} says that if $A$ is abelian, then $(A,+)$ is an abelian group with unit $e$ (it actually states a stronger property: $A$ is polynomially equivalent to a module, whose group reduct is $(A,+)$).
According to \cite[Section 7]{FM}, if $A$ is nilpotent, then $(A,+)$ is a nilpotent loop with unit $e$ (no polynomial equivalence in this case).

\subsection{In groups}

It is instructive to look at how the commutator theory from universal algebra applies to groups. Unlike in loops, the standard commutator in groups is in accordance with the commutator theory. In fact, the situation in groups (and also in rings) gave rise to the general commutator theory. Nevertheless, an elementary proof that the two commutators agree in groups is not obvious, and the reader might want to look at one, for instance, in \cite{MS}.

In groups, unlike in loops, if $A$ is a normal subgroup of $G$ and $A$ itself is an abelian group, then $A$ is abelian in $G$.
It is interesting to see why. The following chain of equivalent conditions settles it: $A$ is abelian in $G$, $[\gamma_A,\gamma_A]=0_G$ (by definition), $[A,A]_G=1$ (\emph{because the two commutators agree}), $A$ is a commutative group (by definition of the commutator in groups), $Z(A)=A$ (by definition of the center in groups), $\zeta(A)=1_A$ (because the two centers agree), $[1_A,1_A]=0_A$, $A$ is abelian (by definition).
This explains why it is safe to call commutative groups by the traditional name abelian groups, and why it is not necessary to distinguish between normal subgroups that are abelian and normal subgroups that are abelian \emph{in} the enveloping group.

\begin{problem}
Investigate the commutator in varieties of loops close to groups, e.g., Moufang loops, Bruck loops or automorphic loops. Describe what does it mean for a subloop of $Q$ to be abelian in $Q$. Is congruence solvability equivalent to solvability here?
\end{problem}

\subsection{Alternative approaches to commutators, nilpotency and solvability}

Each of the following four paragraphs presents an alternative to what we have done here.

Recent discussions in the universal algebraic community seem to lead to a conclusion that the notion of nilpotency coming from the Freese-McKenzie commutator theory is too weak. A new approach, called \emph{supernilpotency}, based on Bulatov's higher commutators, has been promoted recently by Aichinger and Mudrinski \cite{AM}. An important property of supernilpotency, reflecting the situation in finite groups, is the following. A finite algebra with a Mal'tsev term (a finite loop in particular) is supernilpotent if and only if it is a direct product of nilpotent algebras of prime power size. Wright \cite{W} proved that a loop $Q$ satisfies the latter property if and only if $\mlt Q$ is nilpotent. A characterization of infinite supernilpotent loops, and more generally, calculation of higher commutators in the variety of loops, is an interesting open problem.

Yet another approach to nilpotency has been proposed by Mostovoy \cite{Mo}, using so-called \emph{commutator-associator filtration}. The relation between the commutator-associator filtrations and the universal algebraic approach is not clear.

Solvability in loops has been tackled by Lemieux et al. \cite{LMT} in connection with the question whether algebras can express arbitrary Boolean functions. They introduced the notion of \emph{polyabelianess}, a property of loops strictly between nilpotency and solvability (in the Bruck sense), and proved that a finite loop is polyabelian if and only it is \emph{not} able to express Boolean functions. It follows easily from the tame congruence theory \cite{HM} that, for finite algebras with a Mal'tsev term (finite loops in particular), solvability in the sense of commutator theory is equivalent to inability to express Boolean functions in the sense of \cite{LMT}. Hence, for finite loops, polyabelianess is the same as congruence solvability. The relation of the two notions in the infinite case is under investigation.

Finally, let us mention that in category theory, an alternative commutator theory, called the \emph{Huq commutator}, has been proposed. For groups, the two commutators agree. For loops, they do not \cite{HL}. Translating the Huq commutator into loop theory might identify important structural features in loops.

\section*{Acknowledgment}

The proof of Proposition \ref{Pr:InnGens} was suggested to us by A.~Dr\'apal. We acknowledge the assistance of the model builder \texttt{mace4} \cite{Mace4}, the Universal Algebra Calculator \cite{UAC}, and the \texttt{LOOPS} \cite{LOOPS} package for \texttt{GAP} \cite{GAP}---we used them to construct several examples and counterexamples. We thank an anonymous referee of an earlier version of the paper for many useful comments and suggestions for further research. We also thank an anonymous referee of this version for a thorough survey of related literature.

\end{document}